\RequirePackage{luatex85}
\documentclass[12pt]{article}
\usepackage{latexsym}
\usepackage[varumlaut]{yfonts}
\usepackage{epsfig}
\usepackage{amsmath}
\usepackage{amssymb}
\usepackage{amsthm}
\usepackage{color}
\usepackage{tikz-cd}
\usepackage{pb-diagram}
\usepackage[all]{xy}

\newtheorem{theorem}{Theorem}[section]
\newtheorem{corollary}[theorem]{Corollary}

\newtheorem{prop}[theorem]{Proposition}
\newtheorem{lemma}[theorem]{Lemma}

\theoremstyle{remark}
\newtheorem{remark}[theorem]{Remark}

\theoremstyle{definition}
\newtheorem{definition}[theorem]{Definition}

\DeclareMathOperator\diag{diag}
\DeclareMathOperator\id{id}

\DeclareMathOperator\tr{tr}

\usepackage[normalem]{ulem}

\numberwithin{equation}{section}
\newcommand{\SL}{{\rm SL}_3 \mathbb C}
\newcommand{\SLR}{{\rm SL}_3 \mathbb R^{\rm H}}
\newcommand{\SLRn}{{\rm SL}_{n+1} \mathbb R}
\newcommand{\SLRnm}{{\rm SL}_{n} \mathbb R}
\newcommand{\GLR}{{\rm GL}_3 \mathbb R^{\rm H}}
\newcommand{\SLt}{{\rm SL}_2 \mathbb R^{\rm H}}
\renewcommand{\sl}{\mathfrak{sl}_3 \mathbb C}
\newcommand{\slrStan}{\mathfrak{sl}_3 \mathbb R}
\newcommand{\slr}{\mathfrak{sl}_3 \mathbb R^{\rm H}}
\newcommand{\di}{\operatorname{diag}}
\newcommand{\C}{\mathbb C}
\newcommand{\PC}{{\mathbb{C}^{\prime}}}
\newcommand{\R}{\mathbb R}
\newcommand{\Fi}{\operatorname{Fix}}
\newcommand{\Ad}{\operatorname{Ad}}
\newcommand{\CM}{S^{2n+1}_{n+1}}
\newcommand{\CMf}{S^{5}_{3}}
\newcommand{\piH}{\pi_\mathcal{H}}
\newcommand{\f}{\mathfrak{f}}
\newcommand{\DP}{{\mathbb {C}^{\prime}}\!P^n }
\newcommand{\DPt}{\mathbb {C}^{\prime} P^2 }

\newcommand{\D}{\mathbb D}
\newcommand{\LGr}{\operatorname{L_{Gr}^H}}
\newcommand{\SLGr}{\operatorname{SL_{Gr}^H}}
\newcommand{\SO}{{\rm SO}_{3}^{\rm H}}
\renewcommand{\O}{{\rm O}_{3}^{\rm H}}
\newcommand{\so}{\mathfrak{so}_3^H}
\newcommand{\soStan}{\mathfrak{so}_3}

\newcommand{\Uone}{{\rm SO}_2}
\newcommand{\g}{\mathcal {G}}
\newcommand{\X}{\mathfrak{X}}
\newcommand{\Xt}{\mathcal{X}}
\newcommand{\Y}{\mathfrak{Y}}
\newcommand{\Yt}{\mathcal{Y}}
\newcommand{\ip}{i^{\prime}}
\newcommand{\Ip}{J^{\prime}}
\begin{document}

\title{Half-dimensional immersions into the para-complex projective space
 and Ruh-Vilms type theorems}

\author{Josef F. Dorfmeister \thanks{%
Technical Univ.~M\"unchen,
 Fakult\"at f\"ur Mathematik,  TU-M\"unchen,
 Boltzmannstr.3, D-85747,
 Garching,  Germany
 ({\tt dorfm@ma.tum.de}).
}
\and
Roland Hildebrand \thanks{%
Univ.~Grenoble Alpes, CNRS, Grenoble INP, LJK, 38000 Grenoble, France
({\tt roland.hildebrand@univ-grenoble-alpes.fr}).} \thanks{%
Moscow Institute of Physics and Technology, MMO Laboratory, 117303 Moscow, Russia.}
\and
Shimpei Kobayashi \thanks{%
Univ.~Hokkaido, Department of Mathematics, Hokkaido University,
 Sapporo, 060-0810, Japan
({\tt shimpei@math.sci.hokudai.ac.jp})
 He is partially supported by JSPS KAKENHI Grant Numbers JP18K03265 and
 JP22K03304.}
}

\maketitle

\begin{abstract}
In this paper we study isometric immersions
 $f:M^n \to \DP$ of an $n$-dimensional pseudo-Riemannian manifold $M^n$
 into the $n$-dimensional para-complex projective
space $\DP$. We study the immersion $f$ by means of a lift $\f$ of $f$ into a quadric hypersurface in $\CM$. We find the frame equations and compatibility conditions. We specialize these results to dimension $n = 2$ and a definite metric on
 $M^2$ in isothermal coordinates and consider the special cases of
 Lagrangian surface immersions and minimal surface immersions.
 We characterize surface immersions with special properties
 in terms of primitive harmonicity of the Gauss maps.
\end{abstract}

\section{Introduction}
In this paper  we study isometric immersions $f: M^n \to \DP$ of an $n$-dimensional pseudo-Riemannian manifold $M^n$ into the para-complex projective space $\DP$, more precisely into the open dense subset of non-orthogonal pairs in the product $\mathbb {R}P^n \times \mathbb {R}P_n$, 
where $\mathbb {R}P^n$, $\mathbb {R}P_n$ is the real projective space and its dual, respectively. This target space is a para-K\"ahler space form which has been listed in the classification \cite{KaneyukiKozai} of para-K\"ahler symmetric spaces\footnote{The para-complex projective space $\DP$ appears in the example on pp.~92--93 of  \cite{KaneyukiKozai} with parameters $F = \mathbb R$ and $(p,q) = (1,n)$. In that work this space has been represented as the cotangent bundle of the real projective space $\mathbb {R}P^n$, but we shall work with the representation as the mentioned subset of the product $\mathbb {R}P^n \times \mathbb {R}P_n$, because the latter emphasizes the primal-dual symmetry of the space.},
and the para-K\"ahler structure on $\DP$ has been studied in \cite{Hildebrand11B}.

We consider only immersions, whose tangent spaces are transversal to both eigen-distributions $\Sigma^{\pm}$ of the para-complex structure $\Ip$ of $\DP$. We shall call such immersions \emph{non-degenerate}. It has been shown in \cite[Section 4]{Hildebrand11A} that every non-degenerate immersion defines a dual pair of projectively flat torsion-free affine connections $\nabla,\nabla^*$ on $M^n$, and, vice versa, every such pair of connections on $M^n$ defines locally a non-degenerate immersion into $\DP$ which is unique up to the action of the automorphism group of $\DP$.

 Let us explain the results of this paper, section after section.
 In Section \ref{sec:frame_eq} we consider for a given immersion $f$ the relevant geometric objects on $M^n$, namely a dual pair of torsion-free projectively flat connections $\nabla,\nabla^*$, a non-symmetric Codazzi tensor $h$ of type $(0,2),$ whose symmetric part equals the metric $g$ and whose skew-symmetric part $\omega$ measures the deviation of the immersion from a Lagrangian one, and a cubic form $C_{\alpha\beta\gamma} = \nabla_{\gamma}h_{\alpha\beta}$ which is symmetric in the last two indices. We express these structures in terms of a lift $\f$ of $f$ into a quadric hypersurface $\CM$ in the product $\mathbb R^{n+1} \times \mathbb R_{n+1}$ of the real vector space with its dual. We will finally derive the \textit{Maurer-Cartan} equation for an immersion $f$
 as Theorem  \ref{lem:Maurer_Cartan_general}, which will be effectively used in Sections \ref{sec:surface} and \ref{sc:alg}
 in case of surfaces.
 In Section \ref{sec:minimal} we compute the second fundamental form of the immersion $f$ in terms of the cubic form $C$, see Theorem \ref{thm:minimal}. This allows us to specify our results for the case of minimal immersions.

In Sections \ref{sec:frame_eq} and \ref{sec:minimal} we consider the case of general dimension $n$ and arbitrary signature of the metric on $M^n$, while in Sections \ref{sec:surface} and \ref{sc:alg} we specialize to dimension $n = 2$ and a definite metric on $M^2$. For simplicity we assume all immersions to be smooth. In Section \ref{sec:surface} we study definite immersions from a surface $M^2$ into
$\DPt$. Using this immersion we introduce isothermal coordinates on $M^2$ and compute the objects $\nabla,h,C$ as well as the frame equations and compatibility conditions in these coordinates. As a result we obtain the \textit{fundamental
theorem} of definite surfaces in $\DPt$ in Theorem \ref{thm:FTS}.

Minimal Lagrangian surfaces in $\mathbb {C}P^2$ have been considered in many papers, e.g., \cite{MaMa05,DorfmeisterMa16a,DorfmeisterMa16b,Mironov10,Chang2000}.
 In particular in \cite{DoKoMa19}, minimal Lagrangian or minimal surfaces
 have been characterized by various Gauss maps, the so-called
\textit{Ruh-Vilms type theorems}
 have been obtained. In Section \ref{sc:alg}, we will characterize
 surfaces in $\DPt$ with special properties in terms of \textit{primitive harmonic maps},
 which are
 special harmonic maps into $k$-symmetric space $(k\geq 2)$, see Theorem \ref{equivprimitive}.
 In Section \ref{sc:Ruh-Vilms} we will define various Gauss maps for surfaces in
 $\DPt$ by using various bundles over $\CMf$ and finally derive Ruh-Vilms type theorems, Theorem \ref{Thm:3.6}.

 In Appendix \ref{app:basic}, basic results about $\DP$ and $\CM$ will be discussed.
 In Appendix \ref{app:ksymmetric}, $k$-symmetric spaces and primitive harmonic maps
 will be introduced, and finally in Appendix \ref{app:bundles} various bundles will be
 explained.\footnote{The authors feel that the name given to the objects 
 considered may not fit to what algebraic geometers use. 
 We have chosen nevertheless to use the notation used in several papers preceding ours.}

\section{Half-dimensional immersions into the para-complex projective space
form $\DP$}\label{sec:frame_eq}
In this section we derive the frame equations for non-degenerate $n$-dimensional immersions $f$ of a manifold $M^n$ into
 the para-complex projective space $\DP$. The entries of the corresponding Maurer-Cartan forms are expressed by the components of a projectively flat affine connection $\nabla$ and a non-degenerate non-symmetric tensor $h$ which satisfies a Codazzi equation and is an explicit function of the Ricci tensor of $\nabla$
(Theorem \ref{lem:Maurer_Cartan_general}, Remark \ref{remarkRicci}). The components of $h$ and $\nabla$ will in turn be expressed in terms of a lift $\f$ of $f$,
 which exists at least on simply connected charts, into a quadric hypersurface $\CM$ of a real vector space (Lemmas \ref{lem:P_general} and \ref{prop:auxil_coord}).
 In this and the next section we work with
 arbitrary coordinates on $ M^n$.

 \subsection{Para-complex projective space}
 Consider the para-complex projective space
\begin{equation}\label{eq:DP}
 \DP = \{ ([x],[\chi]) \in \mathbb {R}P^n \times \mathbb {R}P_n \,|\,
\langle x, \chi\rangle>0  \},
\end{equation}
where $\mathbb {R}P^n$ is the $n$-dimensional real projective space,
$\mathbb {R}P_n$ is its dual, and some representatives $x \in \mathbb R^{n+1}$ and $\chi \in \mathbb R_{n+1}$ of $[x]$ and $[\chi]$, respectively, are positive with respect to the natural pairing
 $\langle \,,\, \rangle$ between $\R^{n+1}$ and $\R_{n+1}$. Consider $\DP$ as a fibration over $\mathbb {R}P_n$. We have that $\mathbb {R}P_n$ is connected, and that the fiber over a fixed point of $\mathbb {R}P_n$ is exactly an affine chart in $\mathbb {R}P^n$, which is contractible. Hence we even have that the fundamental group of $\DP$ equals that of $\mathbb {R}P_n$, and $\DP$ is isomorphic to one of the reduced para-complex projective spaces in \cite[Theorem 3.1]{GadeaAmilibia92}. In any affine chart on $\DP$ the para-K\"ahler structure is generated by the para-K\"ahler potential $\log|1+\sum_i[x]^i[\chi]^i|$ \cite[Section 4]{Hildebrand11A}. Let us denote the metric on $\DP$,  the symplectic form, and the para-complex structure by
$g$, $\omega$, and $\Ip$,
respectively. Note that $\Ip$ acts
  for $X= (\mathfrak X, \mathcal X) \in T_{([x], [\chi])} \mathbb {R}P^n \times T_{
 ([x], [\chi])} \mathbb {R}P_n = T_{([x], [\chi])} \DP$
by
\[
 \Ip : (\mathfrak X, \mathcal X)  \mapsto (\mathfrak X,-{\mathcal X})
\]
and we have the relations
\[
 g(X,Y) = \omega(\Ip X,Y), \qquad \omega(X,Y) = g(\Ip X,Y).
\]
 The integrable eigen-distributions of $\Ip$  on $\DP$ are denoted by $\Sigma^{\pm}$, respectively. For a survey on para-K\"ahler spaces see, e.g., \cite{EST06}, for a detailed study of the para-K\"ahler space form $\DP$ see \cite{Hildebrand11B}. In Appendix \ref{app:basic}, we shall discuss the para-K\"ahler structure and
 the basic geometry of $\DP$ in detail.

We shall consider immersions $f: M^n \to \DP$ which are transversal to both distributions $\Sigma^{\pm}$. We shall call such immersions \emph{non-degenerate}. The Levi-Civita connection $\widehat \nabla$ of the metric $g$ can be decomposed into a component
 $\nabla$ tangent to $f$ and a component in the distribution $\Sigma^-$, defining a torsion-free affine connection $\nabla$ on $M^n$. In the same way, a torsion-free affine connection $\nabla^*$ can be defined on $M^n$ by decomposing $\widehat \nabla$ into a component tangent to $f$ and a component in $\Sigma^+$. Both connections $\nabla,\nabla^*$ are projectively flat \cite[Lemma 4.2]{Hildebrand11A}. The pullback of the non-symmetric tensor $g + \omega$ defines a non-degenerate non-symmetric tensor $h$ (the para-hermitian form) of type $(0,2)$ on $M^n$ \cite[Lemma 2.3]{Hildebrand11A}, i.e.,
\[
h(X, Y) = (g +\omega)(f_* X, f_*Y),
\]
for tangent vectors $X, Y$ on $M^n$.
 This tensor satisfies respectively the Codazzi equations  and
the duality relation
 \cite[Theorem 2.1]{Hildebrand11A}
\begin{gather} \label{Codazzi}
(\nabla_Xh)(Y,Z) = (\nabla_Zh)(Y,X),\quad (\nabla^*_Xh)(Y,Z) = (\nabla^*_Zh)(X,Z), \\
 Xh(Y,Z) = h(\nabla_XY,Z) + h(Y,\nabla^*_XZ),\label{duality}
\end{gather}
 for all tangent vectors $X,Y,Z$ on $M^n$. From \eqref{duality}, it is easy to see that
 $\widehat \nabla = (\nabla + \nabla^*)/2$. From these equations it follows that the
 \emph{difference tensor} of type $(1, 2)$
\[
 K = \nabla^* - \nabla = 2 (\widehat \nabla - \nabla)
\]
satisfies the relation
\begin{equation} \label{K_C}
(\nabla_X h)(Z, Y)= h(Z, K (X, Y)),
\end{equation}
 where $K(X, Y) = \nabla^*_X Y - \nabla_X Y$.
 It is convenient to introduce a tensor of type
 $(0, 3)$, the \emph{cubic form}
 $C$ by
 \[
  C = \nabla h,
 \]
 and from the Codazzi equation \eqref{Codazzi},
 $C(X, Y, Z) = \nabla_Z h(X, Y)$ is symmetric in the last two indices.

{\remark \label{remarkRicci}
The tensor $h$ can be obtained from the Ricci tensor $\operatorname{Ric}$ of the connection $\nabla$ by \cite[Lemma 4.3]{Hildebrand11A}
\begin{equation} \label{P_Ricci}
h(X, Y) = \frac{1}{n^2-1}\left\{n\operatorname{Ric}(X, Y) + \operatorname{Ric}
 (Y, X)\right\}.
\end{equation}
 On the other hand, if a manifold $M^n$ is equipped with a projectively flat connection $\nabla$ with tensor $h$ given by \eqref{P_Ricci}, then locally there exists an immersion $f: M^n \to \DP$ such that $\Sigma^-$ is transversal to $f$, the tensor $h$ is the pull-back of $g + \omega$ on $M^n$, and $\nabla$ is the affine connection generated by the transversal distribution $\Sigma^-$
 as above \cite[Theorem 4.3]{Hildebrand11A}.
}

\medskip

\subsection{A natural fibration and horizontal lifts}\label{subsc:fibration}
 We shall denote the local coordinates on $M^n$ by $y=(y^1, \dots, y^n)$.
 The coordinates on $\mathbb R^{n+1}$ shall be denoted by $x=(x^1, \dots, x^{n+1})$,
 the coordinates on the dual space $\mathbb R_{n+1}$ by $\chi=(\chi^1, \dots, \chi^{n+1})$.
 The dual pairing on these vector spaces will be denoted
 by $\langle \cdot,\cdot \rangle$.

 As in similar investigations,
 an immersion  $f: M^n \to \DP$  will be discussed via a lift
 into the total space  $\CM$  of a fibration over $\DP$.
 In this paper we will consider as total space the quadric
\begin{equation} \label{quadric}
\CM= \left\{ (x,\chi) \in \mathbb R^{n+1} \times \mathbb R_{n+1} \;|\; \langle x,\chi \rangle = 1 \right\},
\end{equation}
 and will use the projection map
\begin{equation}
\piH : \CM \rightarrow \DP, \quad (x,\chi) \mapsto ([x], [\chi]).
\end{equation}
 Clearly, the tangent space $T_{(x,\chi)} \CM$ to $\CM$
 can be realized by pairs of vectors, $(\widehat{\X}, \widehat{\Xt}) \in \R^{n+1} \times  \R_{n+1}$
 satisfying $\langle \widehat{\X}, \chi \rangle + \langle x, \widehat{\Xt} \rangle = 0$.
 In order to describe tangent vectors to $\DP$, one needs to introduce an equivalence relation for the pair $(\widehat{\X}, \widehat{\Xt})$,
 since $\widehat\X$ and $\widehat\Xt$ are not uniquely defined by the equation just stated.
 More naturally, one can introduce the uniquely defined horizontal distribution
$\widehat{\mathcal{H}}$
 and the vertical distribution $\widehat{\mathcal{V}}$ defined by
\begin{align} \label{xi_eta_chi_pi_orth}
\widehat{\mathcal{H}}_{(x,\chi)}&= \left\{ (\widehat{\X}, \widehat{\Xt}) \mid
  \langle  \widehat{\mathfrak X}, \chi \rangle = 0 \quad\mbox{and} \quad
  \langle x, \widehat{\Xt} \rangle = 0  \right\}, \\
\widehat{\mathcal{V}}_{(x,\chi)} &= \R (x,-\chi).
\end{align}
The horizontal distribution has the following properties. For $(\widehat{\mathfrak{X}},
 \widehat{\mathcal{X}}) \in T_{(x,\chi)} \CM$ we have
 $(\widehat{\mathfrak{X}},
 \widehat{\mathcal{X}}) \in \widehat{\mathcal{H}}_{(x,\chi)}$ if and only if $
(\widehat{\mathfrak{X}},
 - \widehat{\mathcal{X}}) \in \widehat{\mathcal{H}}_{(x,\chi)}$. Moreover, for
$(\widehat{\X},
 \widehat{\Xt}) \in \widehat{\mathcal{H}}_{(x,\chi)}$ we have
 $d\piH(\widehat{\X},  -\widehat{\Xt}) = \Ip d\piH(\widehat{\X},
 \widehat{\Xt})$. This has the following consequence.
\begin{prop}\label{J_horizontal_proposition}
\mbox{}
\begin{enumerate}
\setlength{\itemsep}{0cm}
\renewcommand{\labelenumi}{(\arabic{enumi})}
 \item[{\rm (1)}]
 The projection $\piH : \CM \rightarrow \DP,  (x,\chi) \mapsto ([x], [\chi])$ is a pseudo-Riemannian submersion,
 the differential of which has as  kernel the distribution $\widehat{\mathcal{V}}$
 and is for all $(x, \chi) \in \CM$ an isomorphism from $\widehat {\mathcal{H}}_{(x, \chi)}$ to
$T_{([x], [\chi])} \DP$.

\item[{\rm (2)}]
Let $X,Y \in T_{([x],[\chi])}\DP$ be arbitrary tangent vectors, and
 let $(\widehat{\mathfrak{X}},  \widehat{\mathcal{X}}),
 (\widehat{\mathfrak{Y}},  \widehat{\mathcal{Y}})\in \widehat{\mathcal{H}}_{(x, \chi)}$ be their pre-images under the map $d\piH$ at $(x, \chi)$. Then
\[ (g+\omega)(X,Y) = \langle \widehat{\mathfrak{X}},\widehat{\mathcal{Y}} \rangle.
\]
\end{enumerate}
\end{prop}
 The proof will be given in Appendix \ref{subsub:induced}.

 Let $f: M^n \to \DP$ be an immersion and assume that $f$ has a lift $\f : M^n
 \to \CM$. Here $\f$ is defined by the property that $(x, \chi) = \f(y)$
  projects to $([x],[\chi]) = f(y)$.
 It is easy to see that a  lift is unique up to ``scalings''  of
 the form
 \begin{equation} \label{scaling}
 (x,\chi) \mapsto (\alpha x, \alpha^{-1}\chi)
 \end{equation}
 for never vanishing scalar functions $\alpha$.
 In general, a given $f: M^n \to \DP$  can not be lifted, see
 Proposition \ref{prp:nonlift}.

For later purposes we decompose the tangent map of a lift in more detail:
Let $\f$ be a lift of some immersion $f: M^n \to \DP$ and
write $\f(y) = (x(y), \chi(y))$.
Then the differential of $\f$ can be decomposed in $T_{(x,\chi)} \CM$  in the form
\begin{equation} \label{dectildef}
d_y\f (Z) = (\xi(Z), \eta(Z)) + (\langle d_yx(Z), \chi \rangle x,  \langle x, d_y\chi (Z) \rangle \chi),
\end{equation}
where the first term is the horizontal component and the second term is the vertical component of $d_y\f (Z)$. Moreover, we have
\begin{equation} \label{horizontal_forms}
\xi(Z) = d_yx(Z) - \langle d_yx(Z), \chi \rangle x, \hspace{2mm} \mbox{and} \hspace{2mm}
\eta(Z) = d_y\chi (Z) - \langle x, d_y \chi(Z) \rangle \chi.
\end{equation}
 For convenience we also introduce the 1-form
\begin{equation} \label{psi_def}
\psi (Z) = \langle d_yx (Z),\chi \rangle = -\langle x,d_y\chi (Z) \rangle.
\end{equation}

It is clear that using $\psi$ one can write the vertical component of $d \f(Z)$ as
\[
\psi (Z)( x, -\chi).
    \]
If $\f_\alpha (y) = (\alpha x, \alpha^{-1} \chi)$ is another lift of $f$, then we obtain
\begin{equation} \label{horizontalalpha}
(\xi_\alpha (Z),  \eta_\alpha  (Z)) =(\alpha \xi(Z)  ,\alpha^{-1} \eta (Z))
\end{equation}
for its horizontal component, and for its vertical component we derive
\begin{equation} \label{verticalalpha}
\left(\langle d_y ( \alpha x)(Z), \alpha^{-1} \chi \rangle \alpha x,  \;
 \langle \alpha x,d_y (\alpha^{-1}\chi) (Z) \rangle \alpha^{-1}\chi\right)
=(\psi(Z) + d_y \log|\alpha|) (\alpha x, -\alpha^{-1} \chi).
\end{equation}
Hence under a scaling \eqref{scaling} the form $\psi$ transforms as $\psi \mapsto \psi + d\log|\alpha|$. We may thus add arbitrary differentials to $\psi$ by changing the lift $\f$.

As mentioned above, as a next step and as in similar geometric situations,  one wants to choose not only some lift, but preferably some ``horizontal lift''  for a given immersion $f: M^n \to \DP$.

\begin{definition} \label{def:horizontal}
Let  $f: M^n \to \DP$  be an immersion and $\f: M^n \rightarrow \CM$ be a lift of $f$.
Then $\f$ is called  a ``horizontal lift'' iff the tangent map takes values in the horizontal distribution $\widehat {\mathcal{H}}$.
In other words, a lift is horizontal, iff the vertical component of its differential vanishes identically,
i.e., $\psi = 0$.
\end{definition}

\begin{prop} \label{prop:horizontal}
Assume  $f: M^n \to \DP$  is liftable with lift $\f: M^n \rightarrow \CM$,
$\f(y) = (x, \chi)$. Then there exists a never vanishing scalar function $\alpha$ such that
$\f_\alpha$ is a
horizontal lift of $f: M^n \to \DP$ if and only if the equation
\begin{equation}\label{eq:horizontalcond}
 \alpha^{-1} d \alpha  = - \psi
\end{equation}
has a global solution on $M^n$.
\end{prop}
\begin{proof}
It is easy to verify that in view of equation (\ref{verticalalpha}) the condition on the lift of being horizontal is equivalent to \eqref{eq:horizontalcond}.
\end {proof}


\medskip


\subsection{Non-symmetric Codazzi tensor, connection, frame equations}
We shall now pass to the main goal of this section and compute the non-symmetric tensor $h$ on $M^n$ in terms of the forms $\xi,\eta$.

{\lemma \label{lem:P_general} Let the non-degenerate immersion $f: M^n \to \DP$ be given by means of a lift $\f: M^n \to \CM$ into the quadric \eqref{quadric}. Define the $\mathbb R^{n+1}$- and $\mathbb R_{n+1}$-valued $1$-forms $\xi,\eta$ on $M^n$ by \eqref{horizontal_forms}. Then the pull-back $h$ of the tensor $g + \omega$ from $\DP$ to $M^n$ is given by
\begin{equation} \label{P_formula}
 h(X, Y) = \langle \xi(X),\eta(Y) \rangle
 \end{equation}
 for all tangent vectors $X, Y$ on $M^n$.}
\begin{proof}
The lemma is an immediate consequence of Proposition \ref{J_horizontal_proposition}.
\end{proof}
\begin{prop}\label{prop:auxil_coord}
 Retain the assumptions of {\rm Lemma \ref{lem:P_general}}, and define the $1$-form $\psi$ by \eqref{psi_def}. Then for all tangent vectors $X$ and $Y$ we have
\begin{equation}\label{eq:xietanabla}
\left\{
 \begin{array}{l}
\xi(\nabla_XY) = X(\xi(Y)) - \psi(X) \xi(Y)
 +\langle \xi(Y), \eta(X)\rangle x, \\[0.1cm]
\eta(\nabla^*_XY) = X(\eta(Y)) + \psi(X) \eta(Y)  +\langle \xi(X), \eta(Y)\rangle \chi.
 \end{array}
\right.
\end{equation}
\end{prop}

\begin{proof}
First note that we can rephrase the left-hand side and the first term in
 the right-hand side in \eqref{eq:xietanabla} together as
 \[
(\nabla_X  \xi)(Y) =  X(\xi(Y))- \xi(\nabla_XY) , \quad
(\nabla^*_X  \eta)(Y) = X(\eta(Y)) - \eta (\nabla^*_XY).
\]
 Then using decompositions of the horizontal part and the vertical part of
 $\nabla \xi$ and $\nabla^*  \eta$, we can set
\begin{equation}\label{eq:xietaundetermined}
 (\nabla_X  \xi)(Y)  = a (X) \xi (Y) + b(X, Y) x, \quad
 (\nabla^*_X  \eta)(Y)  = c (X) \eta (Y) + d(X, Y) \chi,
\end{equation}
 where $a$ and $c$ are $1$-forms and $b$ and $d$ are bi-linear maps.
 Taking the inner product of $\chi$ in the first formula and $x$ in
 the second formula, we have
\[
 b(X, Y) = - \langle \xi(Y), \eta(X)\rangle, \quad  d(X, Y)
 = - \langle \xi(X), \eta(Y)\rangle.
\]
 Here we use the relations
  $\langle X \xi(Y), \chi\rangle = - \langle \xi(Y), \eta(X) \rangle$
 and   $\langle X \eta (Y), x\rangle = - \langle \eta(Y), \xi(X) \rangle$,
 respectively, since $\xi$ and $\eta$ are horizontal vectors.
 We now compute $a$. Interchanging $X$ and $Y$ in the first equation of
 \eqref{eq:xietaundetermined}, and subtracting it from the original equation
 we have
\[
 X \xi(Y) - Y \xi(X) - \xi (\nabla_X Y - \nabla_Y X)
 = a(X) \xi(Y) - a(Y) \xi(X) +v(X, Y) x,
\]
 where we set $v (X, Y) = - \langle \xi(Y),\eta (X) \rangle
 + \langle \xi(X), \eta (Y)\rangle$.
 The left-hand side of the above equation can be simplified by
 using the torsion-freeness of $\nabla$ and the definition of
 $\psi$ as
\begin{align*}
X &\xi(Y) - Y \xi(X) - \xi (\nabla_X Y - \nabla_Y X)\\
&= X d_yx(Y) -X( \psi(Y) x ) - Y d_yx(X) +Y( \psi(X) x)
- \xi ([X, Y])\\
&= - \psi(Y) \xi (X) + \psi(X) \xi(Y) +p(X, Y)x,
\end{align*}
where $p(X, Y) x$ denotes the vertical part. Taking the inner product with $\chi$,
$p(X, Y) = v(X, Y)$ follows.
 Therefore we conclude
\[
( a(Y)- \psi(Y)) \xi (X) + (\psi(X) -a (X))  \xi (Y)=0.
\]
 Since the above equation holds for any vector fields $X$ and $Y$, we have
 $a = \psi$. We can compute $c$ similarly. This completes the proof.
\end{proof}

We now rewrite the above relation in \eqref{eq:xietanabla}
 by local expression.
 We first abbreviate the partial derivative by
\begin{equation}\label{eq:abbreviate}
  \partial_{\alpha} = \frac{\partial}{\partial y^{\alpha}},
\end{equation}
 where $y^1, \dots, y^n$ are local coordinates on $M$.

We now use the local expression of the tensor $h$, i.e., $h_{\alpha \beta}
 =h(\partial_\alpha, \partial_\beta)$.
Denote by $h^{\alpha\beta}$ the inverse tensor, i.e., such that $h^{\alpha\beta}h_{\beta\gamma} = h_{\gamma\beta}h^{\beta\alpha} = \delta^{\alpha}_{\gamma}$.
Here we use the Einstein summation convention.
Moreover, we will use the notation
 \[
\quad (\xi_{\alpha}, \eta_{\alpha}) = (\xi(\partial_{\alpha}),
\eta(\partial_{\alpha}))\quad \mbox{and}\quad
\psi_{\alpha} = \psi(\partial_{\alpha}),
 \]
 where $\partial_1, \dots, \partial_n$ are vector fields defined in \eqref{eq:abbreviate}
 and $\xi$ is the horizontal component of a lift $\f(y)= (x(y), \chi(y))$
 defined in \eqref{horizontal_forms}.

 As a consequence of \eqref{P_formula} we obtain by virtue of \eqref{xi_eta_chi_pi_orth} that
\begin{align} \label{formula1}
 \langle \partial_{\alpha} x, \eta_{\beta} \rangle
 &= \langle \xi_{\alpha} + \langle \partial_{\alpha}x,\chi \rangle  x,
 \eta_{\beta} \rangle = h_{\alpha\beta}, \\
\label{formula2}
\langle \partial_{\alpha} \xi_{\beta}, \chi \rangle
 &= -\langle \xi_{\beta},\partial_{\alpha}\chi \rangle
  = -\langle \xi_{\beta},\eta_{\alpha}
   + \langle x, \partial_{\alpha}\chi \rangle \chi \rangle = -h_{\beta\alpha}.
\end{align}

{\corollary \label{coro:nabla_general} Retain the assumptions in
{\rm Proposition \ref{prop:auxil_coord}}.  Then the affine connection $\nabla$ defined on $M^n$ by the transversal distribution $\Sigma^-$ and the dual affine connection $\nabla^*$ defined on $M^n$ by the transversal distribution $\Sigma^+$ are given by
\begin{align} \label{nabla_formula}
\nabla^{\alpha}_{\beta\gamma} &= - \psi_{\gamma}\delta^{\alpha}_{\beta} + h^{\delta\alpha}\langle\partial_{\gamma}\xi_{\beta},\eta_{\delta} \rangle, \\ \label{bar_nabla_formula}
\nabla^{*\alpha}_{\beta\gamma}  &= \psi_{\gamma}\delta^{\alpha}_{\beta} + h^{\alpha\delta}\langle \xi_{\delta},\partial_{\gamma}\eta_{\beta} \rangle.
\intertext{For the difference tensor we obtain}
 K^{\alpha}_{\beta\gamma} &= 2\psi_{\gamma}\delta^{\alpha}_{\beta} + h^{\alpha\delta}\langle \xi_{\delta},\partial_{\gamma}\eta_{\beta}\rangle - h^{\delta\alpha}\langle \partial_{\gamma}\xi_{\beta},\eta_{\delta} \rangle.
\end{align}
}

\begin{proof}
From \eqref{dectildef} we obtain
\begin{equation} \label{dMuX}
\partial_{\mu}x = \xi_{\mu} + \psi_{\mu}x, \quad \partial_{\mu}\chi = \eta_{\mu} - \psi_{\mu}\chi.
\end{equation}
The first relation in \eqref{eq:xietanabla} can be written as
\[ \xi_{\mu}\nabla^{\mu}_{\beta\gamma} = \partial_{\gamma}\xi_{\beta} - \psi_{\gamma}\xi_{\beta} + \langle \xi_{\beta},\eta_{\gamma} \rangle x.
\]
Combining, we obtain
\[ (\partial_{\mu}x - \psi_{\mu}x)\nabla^{\mu}_{\beta\gamma} = \partial_{\gamma}\xi_{\beta} - \psi_{\gamma}(\partial_{\beta}x - \psi_{\beta}x) + \langle \xi_{\beta},\eta_{\gamma} \rangle x,
\]
which yields
\[ \nabla^{\mu}_{\beta\gamma}\partial_{\mu}x = - \psi_{\gamma}\partial_{\beta}x + \partial_{\gamma}\xi_{\beta} + (\nabla^{\mu}_{\beta\gamma}\psi_{\mu} + \psi_{\gamma}\psi_{\beta} + \langle \xi_{\beta},\eta_{\gamma} \rangle)x.
\]
Taking the scalar product with $\eta_{\delta}$ we obtain by virtue of \eqref{formula1} that
\[  \nabla^{\mu}_{\beta\gamma}h_{\mu\delta} = -\psi_{\gamma}h_{\beta\delta} + \langle \partial_{\gamma}\xi_{\beta},\eta_{\delta} \rangle.
\]
Multiplying by $h^{\delta\alpha}$ we get relation \eqref{nabla_formula}.

Relation \eqref{bar_nabla_formula} is obtained similarly. The expression for the difference tensor readily follows.
\end{proof}

{\corollary \label{cor:C} The cubic form is given by
\[ C_{\alpha\beta\gamma} = \partial_{\gamma}h_{\alpha\beta} + 2\psi_{\gamma}h_{\alpha\beta} - \langle \partial_{\gamma}\xi_{\alpha},\eta_{\beta} \rangle - \langle \partial_{\gamma}\xi_{\beta},\eta_{\alpha} \rangle.
\] }
\begin{proof}
The proof is straightforward by using \eqref{nabla_formula}.
\end{proof}

We now go on to deduce the frame equations. We shall define the primal and dual frames as
\[
F =(x,\xi_1,\dots,\xi_n)\quad\mbox{and}\quad F^* = (\chi,\eta_1,\dots,\eta_n), 		
\]
respectively.
Note that under a scaling $(x,\chi) \mapsto (\alpha x,\alpha^{-1}\chi)$ of the lift $\f$ the frames transform as $F \mapsto \alpha F$, $F^* \mapsto \alpha^{-1}F^*$. Therefore the trace-less parts of the Maurer-Cartan forms are invariant under changes of the lift. We have $(F^*)^TF = \diag\left(1,h^T\right)$ and therefore $F^{-1} = \diag\left(1,h^{-T}\right)(F^*)^T$. This yields
\[
 U_{\alpha} := F^{-1}\partial_{\alpha} F = \diag\left(1,h^{-T}\right)\begin{pmatrix} \chi^T \\ \eta^T \end{pmatrix}\begin{pmatrix} \partial_{\alpha}x &
 \partial_{\alpha}\xi \end{pmatrix}
\]
for the Maurer-Cartan form, where $\xi =(\xi_1, \dots, \xi_n)$, $\eta =(\eta_1, \dots, \eta_n)$. Using \eqref{nabla_formula}, \eqref{bar_nabla_formula} and \eqref{formula1}, \eqref{formula2}, \eqref{dMuX} we obtain explicit expressions for the components of $U$. In a similar way we obtain the components of the dual Maurer-Cartan form $U^*_{\alpha} := (F^*)^{-1}\partial_{\alpha} F^*$. Let us list these expressions in the following theorem.

{\theorem \label{lem:Maurer_Cartan_general} Let the non-degenerate immersion $f: M^n \to \DP$ be given by means of a lift $\f: M^n \to \CM$ into the quadric \eqref{quadric}. Define the $\mathbb R^{n+1}$- and $\mathbb R_{n+1}$-valued $1$-forms $\xi,\eta$ on $M$ by \eqref{horizontal_forms}, and the $1$-form $\psi$ by \eqref{psi_def}. Let the affine connection $\nabla$ on $M^n$ be defined by the transversal distribution $\Sigma^-$, and the dual affine connection $\nabla^*$ by the transversal distribution $\Sigma^+$. Define the frames $F = (x,\xi_1,\dots,\xi_n)$ and $F^* = (\chi,\eta_1,\dots,\eta_n)$. Then the Maurer-Cartan forms
\[
\left\{
\begin{array}{l}
 U_{\alpha}
 = F^{-1}\partial_\alpha F
\\
 U^*_{\alpha} = (F^*)^{-1}\partial_\alpha F^*
\end{array}
\right.
\]
 are given by
\[
\left\{
\begin{array}{l}
(U_{\alpha})_{00} = \psi_{\alpha},\quad (U_{\alpha})_{0\gamma} = -h_{\gamma\alpha},\quad (U_{\alpha})_{\beta 0} = \delta^{\beta}_{\alpha},\quad (U_{\alpha})_{\beta\gamma} = \nabla^{\beta}_{\alpha\gamma} + \psi_{\alpha}\delta^{\beta}_{\gamma},
 \\
(U^*_{\alpha})_{00} = -\psi_{\alpha},\quad (U^*_{\alpha})_{0\gamma} = -h_{\alpha\gamma},\quad
(U^*_{\alpha})_{\beta 0} = \delta^{\beta}_{\alpha},\quad (U^*_{\alpha})_{\beta\gamma} = \nabla^{*\beta}_{\alpha\gamma} - \psi_{\alpha}\delta^{\beta}_{\gamma}. \qedhere
\end{array}
\right.
\]
}

Let now $M^n$ be a manifold equipped with an affine connection $\nabla$ and a non-degenerate $(0,2)$-tensor $h$. By the results of \cite[Section 4]{Hildebrand11A} the frame equations are locally integrable for some appropriate $1$-form $\psi$ if and only if $\nabla$ is projectively flat and $h$ is obtained from the Ricci tensor of $\nabla$ by formula \eqref{P_Ricci}. Let us verify this by direct calculation.

The integrability conditions for the frame $F$ are given by
\[
 \partial_{\delta} U_{\alpha}
 -\partial_{\alpha} U_{\delta}
+ [U_{\delta}, U_{\alpha}]  = {\bf 0}.
\]
The upper left corner of this identity yields $-h_{\alpha\delta} + \partial_{\delta}\psi_{\alpha} = -h_{\delta\alpha} + \partial_{\alpha}\psi_{\delta}$. The lower left block is not informative, while the upper right block yields $h_{\beta\delta}\nabla^{\beta}_{\alpha\gamma} + \partial_{\delta} h_{\gamma\alpha} = h_{\beta\alpha}\nabla^{\beta}_{\delta\gamma} + \partial_{\alpha} h_{\gamma\delta}$. Finally, the lower right block yields
\[ -\delta^{\beta}_{\delta}h_{\gamma\alpha} + \nabla^{\beta}_{\delta\epsilon}\nabla^{\epsilon}_{\alpha\gamma} + \partial_{\delta}\nabla^{\beta}_{\alpha\gamma}+ (\partial_{\delta}\psi_{\alpha})\delta^{\beta}_{\gamma} = -\delta^{\beta}_{\alpha}h_{\gamma\delta} + \nabla^{\beta}_{\alpha\epsilon}\nabla^{\epsilon}_{\delta\gamma} + \partial_{\alpha}\nabla^{\beta}_{\delta\gamma} + (\partial_{\alpha}\psi_{\delta})\delta^{\beta}_{\gamma}.
\]
Denoting the Riemann curvature tensor of $\nabla$ by
\[ R^{\beta}_{\gamma\delta\alpha} = \partial_{\delta}\nabla^{\beta}_{\alpha\gamma} - \partial_{\alpha}\nabla^{\beta}_{\delta\gamma} + \nabla^{\beta}_{\delta\epsilon}\nabla^{\epsilon}_{\alpha\gamma} - \nabla^{\beta}_{\alpha\epsilon}\nabla^{\epsilon}_{\delta\gamma}
\]
and the Ricci tensor by $R_{\gamma\alpha} = R^{\delta}_{\gamma\delta\alpha}$, we obtain the compatibility conditions
\begin{equation} \label{compatibility_n}
\left\{
\begin{array}{l}
\partial_{\delta}\psi_{\alpha} - \partial_{\alpha}\psi_{\delta}
 = h_{\alpha\delta} - h_{\delta\alpha},\\[0.1cm]
 \nabla_{\delta}h_{\gamma\alpha} - \nabla_{\alpha}h_{\gamma\delta} = 0,\\[0.1cm]
R^{\beta}_{\gamma\delta\alpha} = \delta^{\beta}_{\delta}h_{\gamma\alpha} - \delta^{\beta}_{\alpha}h_{\gamma\delta} + \delta^{\beta}_{\gamma}(h_{\delta\alpha} - h_{\alpha\delta}).
\end{array}
\right.
\end{equation}
From the last condition it follows by contraction that $R_{\gamma\alpha} = nh_{\gamma\alpha} - h_{\alpha\gamma}$, which is indeed equivalent to \eqref{P_Ricci}. The last condition can then be rewritten as
\[ R^{\beta}_{\gamma\delta\alpha} + \frac{1}{n^2-1} \left\{\delta^{\beta}_{\alpha}(nR_{\gamma\delta} + R_{\delta\gamma}) - \delta^{\beta}_{\delta}(nR_{\gamma\alpha} + R_{\alpha\gamma}) + (n-1)\delta^{\beta}_{\gamma}(R_{\alpha\delta} - R_{\delta\alpha}) \right\} = 0.
\]
On the left-hand side we recognize the \emph{Weyl projective curvature }tensor $W^{\beta}_{\gamma\delta\alpha}$ \cite[eq.~(7p)]{Weyl21} of the connection $\nabla$. The second condition in \eqref{compatibility_n} is the symmetry of the cubic form $C_{\alpha\beta\gamma} = \nabla_{\gamma}h_{\alpha\beta}$ in the last two indices. It implies the closed-ness of the form $\omega_{\alpha\delta} = \frac12(h_{\alpha\delta} - h_{\delta\alpha})$. The first condition in \eqref{compatibility_n} can be written as $d\psi = -2\omega$. If the second condition holds the first condition can locally be satisfied by an appropriate choice of the potential $\psi$.

Thus the integrability conditions on $\nabla$ and $h$ amount to relation \eqref{P_Ricci}, the vanishing of $W^{\beta}_{\gamma\delta\alpha}$, and the symmetry $C_{\alpha\beta\gamma} = C_{\alpha\gamma\beta}$. These are indeed the necessary and sufficient conditions for projective flatness of $\nabla$ \cite[p.~104]{Weyl21}. Moreover, $W^{\beta}_{\gamma\delta\alpha}$ vanishes identically for $n = 2$, and the condition $W^{\beta}_{\gamma\delta\alpha} = 0$ implies the symmetry of $C$ for $n \geq 3$ \cite[p.~105]{Weyl21}.

\section{Second fundamental form and difference tensor} \label{sec:minimal}
In this section we compute the second fundamental form $I\!I$ of a non-degenerate immersion $f$ of a manifold $M^n$ into $\DP$. We show that the immersion is totally geodesic if and only if the cubic form $C$ vanishes, and it is minimal if and only if $C$ is trace-less with respect to the last two indices. The results of this section are valid not only for immersions into the para-K\"ahler space form $\DP$, but for non-degenerate immersions into general para-K\"ahler manifolds ${\mathbb M}^{2n}_n$.

{\theorem \label{thm:minimal} Let $f: M^n \to {\mathbb M}^{2n}_n$
 be an isometric immersion of a pseudo-Riemannian manifold into a para-K\"ahler manifold such that the eigen-distributions $\Sigma^{\pm}$ of the para-complex structure $\Ip$ are transversal to $f$. Let $\nabla,\nabla^*$ be the affine connections defined on $M^n$ by the transversal distributions $\Sigma^-,\Sigma^+$, respectively, and let $K = \nabla^* - \nabla$ be the difference tensor. Let $\Pi_{\pm} = \frac12(\id \pm \Ip)$ be the projections onto $\Sigma^{\pm}$, respectively, and let $\Pi_N$ be the orthogonal projection onto the normal subspace to $f$.
Then the second fundamental form of $f$ is given by
\[
I\!I(X,Y)= \Pi_N\Pi_-f_*K(X,Y)
\]
for all vector fields $X,Y$ on $M^n$. }

\begin{proof}
Let $\widehat \nabla$ be the Levi-Civita connection of the metric $g$ on $\mathbb M^{2n}_n$. Then by definition $I\!I(X,Y) = \Pi_N\widehat\nabla_{f_*X}f_*Y$. Also by definition of $\nabla,\nabla^*$ we have $\Pi_+(\widehat\nabla_{f_*X}f_*Y - f_*\nabla_XY) = 0$ and $\Pi_-(\widehat\nabla_{f_*X}f_*Y - f_*\nabla^*_XY) = 0$, because $\Sigma^- = \ker \Pi_+$, $\Sigma^+ = \ker \Pi_-$. Using $\Pi_+ + \Pi_- = \id$ we obtain
\begin{align*}
\Pi_-f_*K(X,Y) &= \Pi_-f_*\nabla^*_XY - (\id - \Pi_+)f_*\nabla_XY \\&= \Pi_-
 \widehat\nabla_{f_*X}f_*Y - f_*\nabla_XY + \Pi_+\widehat\nabla_{f_*X}f_*Y \\
&= \widehat\nabla_{f_*X}f_*Y - f_*\nabla_XY.
\end{align*}
Applying the projection $\Pi_N$ to both sides we obtain the desired identity.
\end{proof}
 The cubic form is defined as in Section {\rm \ref{sec:frame_eq}} by the relation
 $C= \nabla h$, where $h$ is the pull-back of
 the sum $g + \omega$ to $M^n$ and $\omega$ is the symplectic form of $\mathbb M^{2n}_n$.

\begin{corollary} \label{cor:K_minimal}
 Assume the conditions of Theorem $\ref{thm:minimal}$ and let $g$ be the pseudo-metric on $M^n$.
 Then the immersion $f$ is minimal if and only if
 $\operatorname{Tr}_g K=0$, and it is totally geodesic if and only if $K = 0$.
 Equivalently, the immersion $f$ is minimal if and only if
 $\operatorname{Tr}_g C=0$, and it is totally geodesic if and only if $C = 0$,
 where $C$ is the cubic form.
\end{corollary}
\begin{proof}
Since the immersion $f$ is non-degenerate, the subspaces $\Sigma^{\pm}$ are transversal to both the tangent and the normal subspaces. Therefore the product $\Pi_N\Pi_-f_*$ maps the tangent bundle $TM^n$ bijectively onto the normal bundle. Then by Theorem \ref{thm:minimal} the mean curvature of $f$
is zero if and only if the contraction of the difference tensor with the metric vanishes. Likewise, the second fundamental form vanishes if and only if the difference tensor vanishes. The second part of the assertion follows from \eqref{K_C} and the non-degeneracy of $h$, which in turn follows from the non-degeneracy of the immersion $f$.
\end{proof}

\section{Definite surface immersions} \label{sec:surface}
In this section we specialize to immersions defined on surfaces $M^2$ with definite metric.  We allow both a positive definite and a negative definite metric. For simplicity we assume that the surface is simply connected. We deduce the frame equations and the compatibility conditions in the uniformizing coordinate on the surface $M^2$. We then consider the special cases of Lagrangian immersions and minimal immersions.

\subsection{The Maurer-Cartan form}
 Whatever the sign of the metric, we may introduce a uniformizing complex coordinate
$z = y^1 + iy^2$
on the surface $M^2$ in which the metric takes the form
\[
 g = 2H e^{u}\,|dz|^2,
\]
 where $u: M^2 \to \mathbb R$ is a function of $z$ and $\bar z$,
 i.e., $u=u(z, \bar z)$ and $H = 1$ (the elliptic case) or $-1$ (the hyperbolic case).
 In the corresponding real coordinates $y^1,y^2$ the tensor $h$ takes the form
\[
\begin{pmatrix}
 h_{11} & h_{12}\\ h_{21} & h_{22}
\end{pmatrix}
 = \begin{pmatrix} a & b \\ -b & a \end{pmatrix}
\]
 with $a = 2 H e^u$.
 We first establish relations between the projectively flat connection $\nabla$, the cubic form $C$, and the expressions $\langle \partial_{\alpha}\xi,\eta \rangle$ and similar scalar products in this coordinate system. This will serve to express the Maurer-Cartan form in terms of the two independent entries $a,b$ of $h$ and their derivatives, the two entries of $\psi$, and the 6 independent entries of the cubic form $C$.

We now convert the real coordinates to complex coordinates. The complex canonical basis vectors take the form
\[
\partial_z = \frac12(\partial_1 - i\partial_2),\quad
\partial_{\bar z} = \frac12(\partial_1 + i\partial_2).
\]
For convenience we introduce the complex functions
\begin{gather}\label{eq:rho}
 c = 1 +  i\frac{b}{a} \quad \mbox{and} \quad
 \rho_z = \langle \partial_z x,\chi \rangle
 = \frac12(\psi_1 - i\psi_2).
\end{gather}
We also introduce the \emph{para-K\"{a}hler angle function} $\theta$ by
 \begin{equation}\label{eq:angle}
 \theta = \arctan \left(\frac{b}{a}\right) + \frac{\pi}2 \in
 (0, \pi).
\end{equation}
 It is straightforward to see that
 $\arg c = \theta - \frac{\pi}2 \in\left(-\frac{\pi}2, \frac{\pi}2\right)$.
\begin{remark}
 The complex function $c$ is nowhere zero on a surface $M^2$ since we assume
 that the immersion $f:M^2 \to \DPt$ is definite.\footnote{There is no ``complex point'' compared to the $\C P^2$ case}
\end{remark}

Using the vector-valued 1-forms $\xi,\eta$,
define vector-valued complex functions
\[
\xi_z = \frac12(\xi_1 - i\xi_2) , \quad \eta_z = \frac12(\eta_1 - i\eta_2),
 \quad \xi_{\bar z} = \overline{\xi_{z}}, \quad \eta_{\bar z} = \overline{\eta_{z}}.
\]
By abuse of notation, we denote $\xi_z$, $\eta_z$ and $\rho_z$
 by $\xi$, $\eta$ and $\rho$ respectively.
From \eqref{xi_eta_chi_pi_orth} we then obtain
\begin{equation} \label{zeta_varsigma_orth}
\langle \xi,\chi \rangle = \langle x,\eta \rangle = 0.
\end{equation}
In complex coordinates we then have
\begin{equation} \label{complexP}
h_{zz} = h_{\bar z\bar z} = 0, \quad h_{z\bar z} = \frac12(a+ib) =
 H e^{u} c, \quad h_{\bar zz} = \frac12(a-ib) =  H e^{u}\bar c.
\end{equation}
By Lemma \ref{lem:P_general} we obtain
\begin{equation} \label{zeta_varsigma_products}
\langle \xi,\eta \rangle = \langle \bar \xi,
\bar\eta \rangle = 0,\quad \langle \xi,\bar\eta \rangle =
 He^{u}c,\quad \langle \bar\xi,\eta \rangle =
 H e^{u}\bar c.
\end{equation}
Differentiating these relations, we obtain
\begin{gather*}
 \langle \xi,\partial_z\eta \rangle = - \langle \partial_z\xi,\eta \rangle, \quad \langle \xi,\partial_z\bar\eta \rangle = - \langle \partial_z\xi,\bar\eta \rangle + \partial_z(H e^{u} c), \\
\langle \xi,\partial_{\bar z}\eta \rangle = - \langle \partial_{\bar z}\xi,\eta \rangle, \quad \langle \xi,\partial_{\bar z}\bar\eta \rangle = - \langle \partial_{\bar z}\xi,\bar\eta \rangle + \partial_{\bar z}(H e^{u} c).
\end{gather*}
Relations \eqref{formula1} yield
\begin{equation} \label{complex_formula1}
\langle \partial_z x,\eta \rangle = \langle \partial_{\bar z}x,
\bar\eta \rangle = 0,\qquad \langle \partial_zx,\bar\eta \rangle
 =He^{u}c,\qquad \langle \partial_{\bar z}x,\eta \rangle =H e^{u}\bar c,
\end{equation}
while relations \eqref{formula2} become
\begin{equation} \label{complex_formula2}
-\langle \xi,\partial_z\chi \rangle = \langle \partial_z \xi,\chi \rangle = \langle \partial_{\bar z}\bar
\xi,\chi \rangle = 0,\qquad \langle \partial_{\bar z}\xi,\chi \rangle
 = -H e^{u}c,\qquad \langle \partial_{z}\bar\xi,\chi \rangle
= -H e^{u}\bar c.
\end{equation}
Note also that since the operators $\partial_z,\partial_{\bar z}$ commute, we have
\begin{align} \label{z_xi_commutator}
\langle \partial_z\bar\xi - \partial_{\bar z} \xi,\eta \rangle =& \langle \partial_z\partial_{\bar z} x - \langle \partial_z\partial_{\bar z} x,\chi \rangle x - \langle \partial_{\bar z} x,\partial_z\chi \rangle x - \langle \partial_{\bar z} x,\chi \rangle \partial_z x,\eta \rangle \nonumber\\ &- \langle \partial_{\bar z}\partial_z x - \langle \partial_{\bar z}\partial_z x,\chi \rangle x - \langle \partial_z x,\partial_{\bar z}\chi \rangle x - \langle \partial_z x,\chi \rangle \partial_{\bar z} x,\eta \rangle \nonumber\\ =& \langle - \bar\rho \partial_z x + \rho \partial_{\bar z} x,\eta \rangle = \rho\bar c He^{u}.
\end{align}
We now introduce respectively functions $\phi$ and $Q$ by
\begin{align}\label{eq:phi}
\phi &:= H e^{-u} \langle \partial_{\bar z} \xi, \eta\rangle, \\
Q &:= \langle \partial_{z} \xi, \eta\rangle. \label{eq:Q}
\end{align}
 Note that $\phi d z$ and $Q d z^3$ are
 well-defined as a $1$-form and a $3$-form on $M^2$, respectively.
\begin{remark}\label{rm:choiceoflift}
\mbox{}
\begin{enumerate}
\setlength{\itemsep}{0cm}
\renewcommand{\labelenumi}{(\arabic{enumi})}
 \item
We now observe that if we change a lift $\f=(x, \chi)$ to
 $(\alpha x, \alpha^{-1} \chi)$ by a real-valued function $\alpha$,
 then the real-vectors $\xi_1, \xi_2$,
 $\eta_1$ and $\eta_2$
 change accordingly to $\alpha \xi_1, \alpha\xi_2$,
 $\alpha^{-1}\eta_1$, and $\alpha^{-1}\eta_2$, respectively.
 Therefore the functions $u, c$, and by virtue of \eqref{zeta_varsigma_products} also $\phi$ and $Q$, are independent of
 choice of a lift $\f$, i.e., they are functions depending on $f$ not $\f$.
 \item The \emph{Tchebycheff form} is defined by
 $T_{\alpha} = \frac12C_{\alpha\beta\gamma}g^{\beta\gamma}$. From Corollary \ref{cor:C} we get $C_{zzz} = -2\langle \partial_z\xi,\eta \rangle$, $C_{zz\bar z} = C_{z\bar zz} = -2\langle \partial_{\bar z}\xi,\eta \rangle$, and hence
 \[ Q = -\frac12 C_{zzz}.
 \]
Therefore the cubic form $Q dz^3$ is nothing but the complex
 component of the cubic form $C = \nabla h$. Further in complex coordinates we have $g^{-1} = \frac{2}{a} \begin{pmatrix} 0 & 1 \\ 1 & 0 \end{pmatrix}$ and hence $T_z = \frac1a(C_{zz\bar z} + C_{z\bar zz}) = -\frac4a\langle \partial_{\bar z}\xi,\eta \rangle = -2He^{-u}\langle \partial_{\bar z}\xi,\eta \rangle$,
\[ \phi = -\frac12T_z.
\]
The form $\phi \,d z$ is sometimes called the \emph{mean curvature} $1$-form.
\end{enumerate}
\end{remark}
\medskip
 We are now in a position to formulate the moving frame equations. Instead of the real moving frames $F,F^*$ introduced in Section \ref{sec:frame_eq} we shall consider the complex moving frames
\[ \widetilde{\mathcal F} = (\xi,\bar\xi,x), \qquad \widetilde{\cal F}^* = (\bar\eta,\eta,\chi).
\]
By \eqref{zeta_varsigma_orth} and \eqref{zeta_varsigma_products} the product
$\widetilde {\cal F}^T \widetilde {\cal F}^*$ equals
\[
\widetilde D=
\begin{pmatrix}
H c e^{u}  & 0  &0 \\
0 & H \bar ce^{u}  & 0 \\
0 & 0 & 1
\end{pmatrix},
\]
and $\widetilde {\cal F}^* = \widetilde {\cal F}^{-T} \widetilde D$.
 The Maurer-Cartan forms
 $\widetilde {\cal U}_z :=
 \widetilde {\cal F}^{-1}\partial_z \widetilde {\cal F}$
 and $\widetilde {\cal U}_{\bar z} := \widetilde {\cal F}^{-1}\partial_{\bar z}\widetilde {\cal F}$
are given by
\begin{equation*}
\widetilde {\cal U}_z =
\widetilde D^{-1}
\begin{pmatrix}
\langle \partial_z \xi,\bar\eta \rangle &
\langle \partial_z\bar\xi,\bar\eta \rangle &
\langle \partial_zx,\bar\eta \rangle \\
\langle \partial_z \xi,\eta \rangle &
\langle \partial_z\bar\xi,\eta \rangle &
\langle \partial_zx,\eta \rangle \\
\langle \partial_z \xi,\chi \rangle &
\langle \partial_z\bar\xi,\chi \rangle &
\langle \partial_zx,\chi \rangle
\end{pmatrix},
\quad
\widetilde {\cal U}_{\bar z}=
\widetilde D^{-1}
 \begin{pmatrix}
\langle \partial_{\bar z}\xi,\bar\eta \rangle &
\langle \partial_{\bar z}\bar\xi,\bar\eta \rangle &
\langle \partial_{\bar z}x,\bar\eta \rangle \\
\langle \partial_{\bar z}\xi,\eta \rangle &
\langle \partial_{\bar z}\bar\xi,\eta \rangle &
\langle \partial_{\bar z}x,\eta \rangle \\
\langle \partial_{\bar z}\xi,\chi \rangle &
\langle \partial_{\bar z}\bar\xi,\chi \rangle &
\langle \partial_{\bar z}x,\chi \rangle
\end{pmatrix}.
\end{equation*}
We now compute $\widetilde {\mathcal U}_{z}$ and $\widetilde {\mathcal U}_{\bar z}$
 as follows: Set
\[
 \partial_z \xi = p \xi + q \bar \xi + r x,
\]
 where $p, q$ and $r$ are unknown complex functions to be determined.
 Taking pairing with respect to $\chi$, it is easy to see that
 $r=0$. Moreover, taking pairing with respect to $\eta$ we have
 $\langle \partial_z \xi, \eta \rangle = q \langle \bar \xi, \eta \rangle$,
 and by \eqref{zeta_varsigma_products} and \eqref{eq:Q},
 $q= H \bar c^{-1} e^{-u} Q$ follows. Finally taking pairing with respect to
 $\bar \eta$, we have
\[
 \langle \partial_z \xi, \bar \eta\rangle = p \langle \xi, \bar \eta\rangle.
\]
 Let us compute the left-hand side by taking the derivative of
 $\langle \xi, \bar \eta\rangle = H e^{u} c $ with respect to $z$, that is,
\[
 \langle \partial_z \xi, \bar \eta \rangle
= -  \langle \xi, \partial_z \bar \eta \rangle + H e^{u}
 c\> \partial_z (\log c +u).
\]
 Since $\partial_z \bar \eta = \partial_z \partial_{\bar z}\chi -
 \partial_z (\bar \rho \chi)$ and by virtue of \eqref{complex_formula2} $\langle \xi, \partial_z\chi \rangle=0$, we get
\[
 \langle \xi, \partial_z \bar \eta\rangle  =
  \langle \xi, \partial_{\bar z} \partial_z \chi \rangle
 = - \langle \partial_{\bar z}\xi, \partial_z \chi \rangle.
\]
 Moreover,  $\partial_z \chi = \eta - \rho \chi$ and thus
 $\langle \partial_{\bar z}\xi, \partial_z \chi \rangle = H e^{u} \phi
 + \rho H e^{u} c$ holds. Therefore
 $p = \frac{\phi}{c} + \rho + \partial_z (\log c + u)$ follows.
 Similarly, set
 \[
 \partial_{z} \bar \xi = p \xi + q \bar \xi + r x,
\]
 where $p, q$ and $r$ are unknown complex functions to be determined.
 Taking pairing with respect to $\chi$, it is easy to see that
 $r= - H e^{u} \bar c$ by \eqref{complex_formula2}.
 Next taking pairing with respect to $\eta$ we have
 $\langle \partial_z \bar \xi, \eta \rangle = q \langle \bar \xi, \eta \rangle = qHe^u\bar c$,
 and by \eqref{z_xi_commutator}, $q= \rho + \bar c^{-1} \phi$ follows.
 Finally taking pairing with respect to  $\bar \eta$, we have
 $\langle \partial_z \bar \xi, \bar \eta\rangle = p \langle \xi, \bar \eta\rangle$,
 and by \eqref{eq:phi} and \eqref{complex_formula2},
 $p = c^{-1} \bar \phi$ follows.
By \eqref{dMuX} we have
 \[
 \partial_{z} x = \xi + \rho x = 1 \cdot \xi + 0 \cdot \bar\xi + \rho \cdot x.
\]
 One can compute $\widetilde{\mathcal U}_{\bar z }$ similarly.

 Thus the Maurer-Cartan form can be computed as follows:
\begin{align*}
\widetilde {\cal U}_{z}&=
\begin{pmatrix}
\rho+c^{-1} \phi + \partial_z (u+ \log c) & c^{-1} \bar \phi & 1 \\[0.1cm]
H \bar c^{-1} e^{- u} Q  & \rho + \bar c^{-1} \phi & 0 \\ 0 &
 -H \bar ce^{u} & \rho
\end{pmatrix}, \\
\widetilde {\cal U}_{\bar z}&=
 \begin{pmatrix}
 \bar \rho+ c^{-1} \bar \phi & H c^{-1}e^{-u}\bar Q & 0 \\[0.1cm]
\bar c^{-1}\phi &\bar \rho+\bar c^{-1} \bar \phi + \partial_{\bar z} (u+ \log \bar c) & 1 \\ -H c e^{u} & 0 & \bar \rho \end{pmatrix}.
\end{align*}
The compatibility conditions
\[
[\widetilde{\cal U}_z,\widetilde {\cal U}_{\bar z}] +
 \partial_z \widetilde {\cal U}_{\bar z}- \partial_{\bar z} \widetilde {\cal U}_{z} = 0
\]
 amount to the real equation $\partial_{\bar z}\rho - \partial_z \bar\rho
 = H e^{u}( c - \bar c)$, which can be written as
\begin{equation} \label{comp_rho}
\operatorname{Im}(\partial_{\bar z}\rho) = H e^{u}
 \operatorname{Im} c
\end{equation}
and is equivalent to the first equation in \eqref{compatibility_n},
and the two compatibility conditions
\begin{gather} \label{complex_comp1}
|c|^{-2} |\phi|^2- |c|^{-2} e^{-2 u} |Q|^2  + H (\bar c- 2 c) e^{u}
 -
\partial_z \left(c^{-1} \bar  \phi\right) +\partial_{\bar z} (c^{-1} \phi)
- \partial_z \partial_{\bar z} (\log c + u) = 0, \\
\label{complex_comp2}
(\bar c^{-1}  - c^{-1})(e^{u} \bar \phi^2-\bar Q \phi )
 - e^{u} \bar \phi \, \partial_{\bar z} \log |c|^2
 + e^{u} (\partial_{\bar z} \bar \phi - \bar \phi \partial_{\bar z} u )
 + \partial_{\bar z}Q =0.
\end{gather}

 As pointed out in (1) in Remark \ref{rm:choiceoflift}, the functions
 $u, c, \phi$ and $Q$ are independent of the choice of a lift $\f$.
 On the other hand, the function $\rho$ depends on the choice of a lift.
\begin{prop}\label{prp:rho0}
 By choosing a lift $\f$ properly, the function $\rho$
 can be made to satisfy the condition
 \begin{equation}\label{eq:rho0}
  \partial_{\bar z} \rho = H c e^u.
 \end{equation}
 In this case, we denote it by $\rho_0$ instead of $\rho$.
\end{prop}
 \begin{proof}
 Let $\rho_0$ as in \eqref{eq:rho0}. Then the compatibility condition \eqref{comp_rho}
 is equivalent to
  \[
   (\rho_0-\rho)_{\bar z} - \overline{(\rho_0-\rho)_{\bar z}}=0.
  \]
 Therefore, the $1$-form
\[
  \Omega = \left\{ (\rho_0-\rho) dz + \overline{(\rho_0-\rho)} d \bar z\right\}
\]
 is a real-closed $1$-form. Let $\delta : \mathbb D \to \R$
 denote a solution to $d \delta = \Omega$. Now the new lift
 $\tilde \f = e^{\delta} \f$ satisfies $\tilde \rho=\rho_0$.
 \end{proof}

We now gauge the frames  $\widetilde {\mathcal F}$ and  $\widetilde {\mathcal F}^*$ to:
\begin{align}
 \mathcal F=  \widetilde {\mathcal F}D , \quad
 \mathcal F^* = \widetilde {\mathcal F}^*D, \label{eq:mathcalF}
\end{align}
 with
\begin{equation*}
D = \widetilde D^{-1/2} \diag(1, 1, i)=
 \begin{pmatrix} (Hc)^{-1/2} e^{-u/2}&0& 0 \\ 0&(H\bar c)^{-1/2} e^{-u/2}&0 \\ 0 &0& i\end{pmatrix}.
\end{equation*}
 Then a straightforward computation shows that the Maurer-Cartan form of
 $\mathcal F$ can be computed as
\[
 \mathcal F^{-1} d \mathcal F :=
 \mathcal U_z d z +  \mathcal U_{\bar z} d \bar z,
\]
with
\begin{equation}\label{eq:UV}
\left\{
\begin{array}{l}
{\cal U}_z =
\begin{pmatrix}
\rho+\frac{1}{2}\partial_z u +\frac12\partial_z \log c + c^{-1} \phi  &
 |c|^{-1} \bar \phi & i (Hc)^{1/2}e^{u/2} \\[0.1cm]
|c|^{-1} e^{- u} H Q  & \rho-\frac{1}{2}\partial_z u -\frac12\partial_z \log \bar c + \bar c^{-1} \phi & 0 \\ 0 &
 i(H \bar c)^{1/2}  e^{u/2} & \rho
\end{pmatrix}, \\[0.8cm]

{\cal U}_{\bar z} =
\begin{pmatrix}
 \bar \rho- \frac12 \partial_{\bar z} u - \frac12\partial_{\bar z} \log c+ c^{-1} \bar \phi & |c|^{-1}e^{-u}H \bar Q & 0 \\[0.1cm]
|c|^{-1}\phi &\bar \rho  + \frac12\partial_{\bar z} u+ \frac12 \partial_{\bar z} \log \bar c + \bar c^{-1} \bar \phi & i(H \bar c)^{1/2} e^{u/2} \\ i(H c)^{1/2} e^{u/2} & 0 & \bar \rho \end{pmatrix}.
\end{array}
\right.
\end{equation}
 We now summarize the
 above discussion as  the following theorem.
\begin{theorem}[Fundamental Theorem of definite surfaces in $\DPt$]
\label{thm:FTS}
 Let $f: M^2 \to \DPt$ be a liftable immersion and $\f:
 M^2 \to \CMf$ a lift. Let $g = 2 H e^{u} d z d \bar z, \, (H \in \{-1, 1\})$
 denote the induced metric, $\theta: M^2 \to (0, \pi)$ the para-K\"ahler angle,
 $Q dz^3$ the cubic form, $\phi dz$ the mean curvature form. Set $c$ by $
 c = 1 + i\tan (\theta - \pi/2)$ and
 $\rho$ by \eqref{eq:rho}. Then \eqref{comp_rho}, \eqref{complex_comp1},
 and \eqref{complex_comp2} are satisfied.

 Conversely let $g= 2 H e^{u} dz d \bar z, \, (H \in \{-1, 1\})$
 be a positive or negative
 definite metric on a simply connected Riemann surface $\mathbb D$.
 Let $\theta : \mathbb D \to (0, \pi)$ be a real valued function and
 $\phi dz$ and $Q dz^3$ be a $1$-form and a $3$-form, respectively.
 Set $c$ by $c = 1 + i \tan (\theta-\pi/2)$ and
 $\rho$ by \eqref{comp_rho}. If these data satisfy \eqref{complex_comp1}
 and \eqref{complex_comp2}, then there exists an immersion $f : \mathbb D
 \to \DPt$ which has invariants stated as above and is unique
 up to isometries of $\DPt$.
\end{theorem}

\subsection{Lagrangian surface immersions}\label{sbsc:Lag}

In this section we specify our results to the case of Lagrangian surface immersions into $\DPt$. If the immersion $f$ is Lagrangian, then $\omega = 0$, $h = g$, and $C = \nabla g$ is totally symmetric. The converse implication also holds.

{\lemma \label{lem:cubic_sym} Let $f: M^2 \to \DPt$ be a non-degenerate surface immersion with totally symmetric cubic form. Then $f$ is Lagrangian. }

\begin{proof}
The condition that $C$ is totally symmetric is equivalent to the condition $\nabla\omega = 0$. Suppose for the sake of contradiction that $\omega$ does not vanish on some neighbourhood $U \subset M$. Then $\nabla$ preserves a non-trivial volume form on $U$ and is equi-affine. This implies that its Ricci tensor is symmetric \cite[Proposition I.3.1]{NomizuSasaki}. But then \eqref{P_Ricci} implies $\omega = 0$, leading to a contradiction.
\end{proof}

The Lagrangian condition $\omega = 0$, or equivalently $b = 0$, has the implication
\begin{gather*}
c = 1.
\end{gather*}
The compatibility condition \eqref{comp_rho} becomes $\operatorname{Im}
 \partial_{\bar z}\rho = 0$. It  is equivalent to the form $\psi$ to be closed. Since $M$ is 
 simply connected,  there exists a real potential $\upsilon$ such that $\psi= d\upsilon$, or $\rho = \partial_z \upsilon$. By an appropriate scaling of the lift $\f$ of $f$ into $\CMf$ we may choose $\upsilon$ equal to any desired smooth real function. In particular, we may achieve
\[
 \rho = 0
\]
by an appropriate choice of the lift. Such a lift is \emph{horizontal} in the sense of Definition \ref{def:horizontal}. In this case the lift $\f$ defines a dual pair of centro-affine immersions $x,\chi$ into $\mathbb R^3$ and $\mathbb R_3$, respectively, whose metric coincides with $g$ and whose centro-affine connection coincides with $\nabla$ \cite[Theorem 4.1]{Hildebrand11B}. 

Conditions \eqref{complex_comp1} and \eqref{complex_comp2} simplify to
\begin{gather*}
\left\{
\begin{array}{l}
\partial_z \partial_{\bar z} u
-|\phi|^2 +e^{-2 u} |Q|^2  + H e^{u}
 =\partial_{\bar z} \phi- \partial_z \bar  \phi, \\
e^{u} (\partial_{\bar z} \bar \phi - \bar \phi \partial_{\bar z} u )
 + \partial_{\bar z}Q =0.
\end{array}
\right.
\end{gather*}
The left-hand side in the first equation is real, while the right-hand side is imaginary. Hence both sides must equal zero, and we have
\begin{equation}\label{complex_comp123-Lag}
\left\{
\begin{array}{l}
 - \partial_z \bar  \phi +\partial_{\bar z} \phi=0,  \\
\partial_z \partial_{\bar z} u-|\phi|^2+ e^{-2 u} |Q|^2  + H e^{u}
 =   0,  \\
e^{u} (\partial_{\bar z} \bar \phi - \bar \phi \partial_{\bar z} u )
 + \partial_{\bar z}Q =0.
\end{array}
\right.
\end{equation}
The Maurer-Cartan forms \eqref{eq:UV} simplify to
\begin{equation}\label{eq:LagUV}
{\cal U}_z =
\begin{pmatrix}
\frac{1}{2}\partial_z u  + \phi  &
 \bar \phi & iH^{1/2}e^{u/2} \\[0.1cm]
 e^{- u} HQ  & -\frac{1}{2}\partial_z u + \phi & 0 \\ 0 &
 iH^{1/2}  e^{u/2} & 0
\end{pmatrix}, \qquad
{\cal U}_{\bar z} =
\begin{pmatrix}
 - \frac12 \partial_{\bar z} u +\bar \phi & e^{-u}H \bar Q & 0 \\[0.1cm]
\phi &\frac12\partial_{\bar z} u + \bar \phi & iH^{1/2} e^{u/2} \\ iH^{1/2} e^{u/2} & 0 & 0 \end{pmatrix}.
\end{equation}
\begin{remark}
From
$ - \partial_z \bar  \phi +\partial_{\bar z} \phi=0$
the $1$-form $\phi d z + \bar \phi d \bar z$ is closed, or equivalently
the Tchebycheff form $T$ is closed for a Lagrangian immersion $f$.
\end{remark}

\subsection{Minimal surface immersions}\label{sbsc:minimal}
In this section we specify our results to minimal surface immersions into
 $\DPt$.

By Corollary \ref{cor:K_minimal} the immersion $f$ is minimal if and only if the Tchebycheff form $T$ vanishes if and only if the function $\phi$ vanishes.
Setting $\phi = 0$ in the compatibility condition \eqref{complex_comp2} gives
 $\partial_{\bar z} Q = 0$, and $Q$ is a holomorphic function. Setting $\phi = 0$ in \eqref{complex_comp1} gives
\[
- |c|^{-2} e^{-2 u} |Q|^2 - \partial_z \partial_{\bar z} u  =
\partial_z \partial_{\bar z} (\log c) -H (\bar c- 2 c) e^{u}.
\]
The left-hand side of the equation is real, and so must be the right-hand side.
Therefore we have an additional equation
 for the para-K\"{a}hler angle $\theta$ in \eqref{eq:angle},
\[
\partial_z \partial_{\bar z}\theta
= 3 H e^{u}\cot \theta.
\]
 The Maurer-Cartan form can be simplified to
\begin{equation}\label{eq:MinUV}
\left\{
\begin{array}{l}
{\cal U}_z =
\begin{pmatrix}
\rho+\frac{1}{2}\partial_z u +\frac12\partial_z \log c &
 0 & i(Hc)^{1/2}e^{u/2} \\[0.1cm]
|c|^{-1} e^{- u} H Q  & \rho-\frac{1}{2}\partial_z u -\frac12\partial_z \log \bar c & 0 \\ 0 &
 i(H \bar c)^{1/2}  e^{u/2} & \rho
\end{pmatrix}, \\[0.8cm]

{\cal U}_{\bar z} =
\begin{pmatrix}
 \bar \rho- \frac12 \partial_{\bar z} u - \frac12\partial_{\bar z} \log c & |c|^{-1}e^{-u}H \bar Q & 0 \\[0.1cm]
0 &\bar \rho  + \frac12\partial_{\bar z} u+ \frac12 \partial_{\bar z} \log \bar c  & i(H \bar c)^{1/2} e^{u/2} \\ i(H c)^{1/2} e^{u/2} & 0 & \bar \rho \end{pmatrix}.
\end{array}
\right.
\end{equation}

\subsection{Minimal Lagrangian surface immersions}
Combining Section \ref{sbsc:Lag} and Section \ref{sbsc:minimal},
we obtain the following equations for a definite
 minimal Lagrangian surface immersion in $\DPt$:
\begin{gather}\label{eq:minLagstr}
\left\{
\begin{array}{l}
\partial_z \partial_{\bar z} u +e^{-2 u} |Q|^2  + H e^{u}
 =0, \\
\partial_{\bar z}Q =0.
\end{array}
\right.
\end{gather}
 Moreover, the Maurer-Cartan form can be simplified to
\begin{equation}\label{eq:minLagUV}
{\cal U}_z =
\begin{pmatrix}
\frac{1}{2}\partial_z u    & 0 & i H^{1/2}e^{u/2} \\[0.1cm]
 e^{- u} HQ  & -\frac{1}{2}\partial_z u  & 0 \\ 0 &
 i H^{1/2}  e^{u/2} & 0
\end{pmatrix}, \qquad
{\cal U}_{\bar z} =
\begin{pmatrix}
 - \frac12 \partial_{\bar z} u  & e^{-u}H \bar Q & 0 \\[0.1cm]
0 &\frac12\partial_{\bar z} u  & i H^{1/2} e^{u/2} \\ iH^{1/2} e^{u/2} & 0 & 0 \end{pmatrix}.
\end{equation}
 The first equation in \eqref{eq:minLagstr} is known as
the \textit{Tzitz\'{e}ica equation} and the Maurer-Cartan form
is identical to that of \cite[Section 4.4 with
 the spectral parameter $\lambda = \pm 1$]{DFKW19}. Therefore it is easy to see
 that
\begin{theorem}
 A  definite minimal Lagrangian immersion in $\DPt$
 defines a definite proper affine sphere
 in $\R^3$ and vice versa.
\end{theorem}
\section{Primitive maps and immersions with special properties}\label{sc:alg}
 In this section we characterize surface immersions in $\DPt$ with special properties
 (minimal, Lagrangian or minimal Lagrangian surfaces) in terms of
 \textit{primitive harmonic maps}. Since the results in this section are an adaptation
 the results of \cite{DoKoMa19} to the case of surface immersions into $\DPt$,
 we will omit detailed proofs, and refer to Appendix \ref{app:ksymmetric}.

\subsection{The real form $\tau$}
 It is easy to see that the determinant of the moving frame $\mathcal F$ in
 \eqref{eq:mathcalF} can be computed as
\[
 \det \mathcal F  =i H^{-1}|c|^{-1} e^{-u} \det \widetilde {\mathcal F}
 = 2 H^{-1}|c|^{-1} e^{-u} \det (\xi_1, \xi_2, x),
\]
 where $\xi_1$ and $\xi_2$ are real-valued vectors as in \eqref{horizontal_forms}.
 Therefore $\det \mathcal F$ takes values in $i \mathbb R^{\times}$.
 Let us  denote $\det \mathcal F$ by $\delta$ with
 a non-vanishing real function $\delta$. Then,
 it is also easy to see that $\det \mathcal F^* =  \delta^{-1}$.

 As discussed in Remark \ref{rm:choiceoflift}, if we change a lift $(x, \chi)$ to
 $(\delta^{1/3} x, \delta^{-1/3} \chi)$, then the real-vectors $\xi_1, \xi_2$,
 $\eta_1$, and $\eta_2$ change accordingly to $\delta^{1/3}\xi_1, \delta^{1/3}\xi_2$,
 $\delta^{-1/3}\eta_1$, and $\delta^{-1/3}\eta_2$, respectively.
 Then $\det \mathcal F = \det \mathcal F^* = 1$ in this
 particular lift. Therefore we have the following.
\begin{lemma}\label{det=1}
 Choosing the initial condition
 of $\mathcal F$ and $\mathcal F^*$ properly, the gauged moving frames
\begin{equation}\label{eq:G}
 \Ad (R_H) (\mathcal F)  \quad \mbox{and}\quad \Ad (R_H^{-T})(\mathcal F^*) , \quad \mbox{with}\quad
 R_H=\begin{pmatrix}
\frac1{\sqrt{2}} & \frac{1}{\sqrt{
2}}&0 \\
\frac{i}{\sqrt{2}}&-\frac{i}{\sqrt{2}}&0 \\
 0&0& \sqrt{-H}
 \end{pmatrix},
\end{equation}
 take values in $\mathrm{SL}_3 \mathbb R$.
 \end{lemma}
\begin{remark}\label{rm:rho}
 The function $\rho$ in \eqref{eq:rho} of the lift $\f$
 such that $\det \mathcal F = 1$ cannot satisfy \eqref{eq:rho0} in general.
 In the following, we normalize a lift $\f$ such that the frame $\mathcal F$ satisfies
 $\det \mathcal F =1$, and we do not assume \eqref{eq:rho0}.
\end{remark}
We define a real
 Lie group
\begin{equation}\label{eq:SLRt}
 \left\{
 A
 \mid
 \Ad (R_H) (A) \in \mathrm{SL}_3 \mathbb R, \;\; \mbox{where $R_H$ is defined in \eqref{eq:G}}
\right\},
\end{equation}
 which is isomorphic to the standard $\mathrm{SL}_3 \mathbb R$,
 and we denote it  by $\SLR$. More explicitly, the Lie group
 $\SLR$ in \eqref{eq:SLRt} can
 be represented by
\begin{equation}\label{eq:SLR}
 \SLR=\left\{
 A = \begin{pmatrix}  a & b & \sqrt{-H} c \\ \bar b & \bar a & \sqrt{-H} \bar c \\
 \sqrt{-H} d &   \sqrt{-H} \bar d & e  \end{pmatrix}\;\Big|\;
 a, b, c ,d, e \in \C\;\;\mbox{and $\det  A =1$}
\right\}.
\end{equation}
 The Lie algebra of the above $\SLR$,
 which is isomorphic to
 the standard Lie algebra $\slrStan$, 
 can be represented by
\[
 \slr=\left\{ A = \begin{pmatrix}  a & b & \sqrt{-H} c \\ \bar b & \bar a & \sqrt{-H} \bar c \\
 \sqrt{-H} d & \sqrt{-H} \bar d & e  \end{pmatrix}\;\Big|\;
 a, b, c ,d, e \in \C\;\;\mbox{and $\tr  A =0$}
\right\}.
\]
 Therefore
 without loss of generality the moving frame $\mathcal F$ (and
 $\mathcal F^*$) of an immersion $f: M^2 \to\DPt$
 takes values in $\SLR$. In the following
 consideration, we always assume this.
 Moreover, one can think $\mathfrak g^{\R} = \slr$ as the real form of $\mathfrak g = \sl$
 given by the anti-linear involution
\begin{equation}\label{eq:tau}
 \tau (X) =\Ad(P_H) \bar X, \quad X \in \sl,
 \quad  P_H =
\begin{pmatrix}
 0 & 1 & 0 \\
 1 & 0 & 0 \\
 0 & 0 & -H
\end{pmatrix}.
\end{equation}
 Note that $P_H$ is given by
\[
R_H^T  R_H = P_H.
\]
 We consider the anti-linear involution $\tau^G$ on
 the group level $G = \SL$ as
\begin{equation}\label{eq:tauonG0}
 \tau^G (g) = \Ad(P_H)(\bar g), \quad g \in \SL.
\end{equation}
 By abuse of notation we will also write $\tau^{G}$ by $\tau$.

\subsection{Primitive maps and immersions with special properties}
 We now consider the order $6$ outer automorphism
 $\sigma$ on $\sl$ is given by
\begin{equation} \label{defsigma}
\sigma_H ( X) =   - P_H^{\epsilon} X^T P_H^{\epsilon},  \quad \mbox{where} \quad
P_H^{\epsilon} =
\begin{pmatrix}
0& \epsilon^2 & 0\\
\epsilon^4& 0 & 0\\
0&0&-H\\
\end{pmatrix}
\end{equation}
 with $\epsilon = e^{\frac{ i\pi}{3}}$. Note that
\[
P_H^{\epsilon} = \diag (\epsilon^2, \epsilon^4, 1) P_H,
\]
and $\sigma_H$ with $H=-1$ has been used in \cite{DoKoMa19}, \cite{DFKW19}, \cite{DK21}. It is easy to see that $\sigma_H$ commutes with
$\tau$ and thus $\sigma_H$ defines a $k$-symmetric space with $k=6$, see
Definition \ref{def:k-symmetric}.
Moreover, instead of $\sigma_H$, one can use $\sigma_H^2$ and
 $\sigma_H^3$, which  are the order $3$ and $2$ automorphisms on $\sl$,
and there are corresponding $k$-symmetric spaces with $k=3$ and $2$,
respectively.
 Then one can introduce the \textit{primitive harmonic} into a
 $k$-symmetric space relative to $\sigma_H$, $\sigma_H^2$ and $\sigma_H^3$,
 respectively, see Definition \ref{def:primitiveharmonic}.
 The following characterizations of surface immersions with special
 properties in terms of primitive harmonic maps are
 verbatim to the case of those of $\mathbb C P^2$, \cite[Theorem 2.4]{DoKoMa19}, thus we will omit  the proof.
\begin{theorem} \label{equivprimitive}
Let $G = \SL$ and $\mathfrak{g} = \sl$ its Lie algebra.
Let $\tau$ denote the real form involution of $G$ singling out
 $G^{\R}=\SLR$ in $G$
and let $\sigma_H = \sigma_H^G$ be the automorphism of order $6$ of $G$ given by
$\sigma_H (g) =   P_H^{\epsilon} (g^{T})^{-1} P_H^{\epsilon}$  in  \eqref{defsigma}.
Assume moreover, that $\f$ is the lift of  a liftable immersion $f$ into $\DPt$ and with frame  $\mathcal F$ in $G^{\R}$.
Then the following statements hold:
\begin{enumerate}
\setlength{\itemsep}{0cm}
\renewcommand{\labelenumi}{(\arabic{enumi})}
 \item[{\rm (1)}] $\mathcal F$ is primitive harmonic relative to $\sigma_H$ if and only if $f$ is minimal Lagrangian in $\DPt$.

\item[{\rm (2)}] $\mathcal F$ is primitive harmonic relative to $\sigma_H^2$
 if and only if $f$ is minimal in $\DPt$.

\item[{\rm (3)}] $\mathcal F$ is primitive harmonic relative to $\sigma_H^3$ if and only if either $f$ is minimal Lagrangian or  $f$ is flat homogeneous in $\DPt$.
\end{enumerate}
\end{theorem}

\section{Ruh-Vilms type theorems}\label{sc:Ruh-Vilms}
 In the following sections, we use the para-hermitian inner product
 of the $3$-dimensional  para-complex vector space $\PC^3$
 with a para-Hermitian form
\begin{equation}\label{eq:hinn}
 \langle u, v \rangle_h = u^{*T} P_H v, \quad
\end{equation}
 where $*$ denotes the para-complex conjugate of a paracomplex vector in $\PC^3$, and $P_H$ is defined in \eqref{eq:tau},
 see also Appendix \ref{subsubsc:para}.

\begin{remark}
 The para-Hermitian form  \eqref{eq:hinn}
 is invariant under the Lie group $\SLR$ defined in \eqref{eq:SLR}.
 It is different from the standard para-Hermitian inner product in
 \eqref{eq:para-inner}, but they are isomorphic.
 The para-Hermitian form in \eqref{eq:hinn}
 is suitable for Ruh-Vilms type theorems.
\end{remark}

 The $3$-dimensional  para-complex vector space $\PC^3$ is
 a symplectic vector space with the symplectic form
 $\omega = -\Im \langle \;,\;  \rangle_h$.

 In \cite[Section 3]{DoKoMa19},
 three $6$-symmetric spaces of dimension $7$ which are bundles
 over $S^5$ were defined, which were $FL_1$, $FL_2$ and
 $FL_3$. We will analogously define bundles over $\CMf$, which will
 be denoted by $FL_1^H, FL_2^H$ and  $FL_3^H$, respectively.
 A detailed construction can be found in Appendix \ref{app:bundles}.

 \subsection{Projections from various bundles}
 A family of (real) oriented Lagrangian subspaces of  $\PC^3$
 forms a submanifold of the manifold of real Grassmannian $3$-spaces of $\PC^3$,
 which will be called the \emph{Grassmannian manifold} of
 oriented Lagrangian subspaces and will be denoted by
 $\LGr(3, \PC^3)$.  It is easy to see that
 $\LGr(3, \PC^3)$ can be represented as the homogeneous space
 $\GLR / \O$. In particular the orbit of $\SLR$ through the point $e \in \SO$ will be
 called the \textit{special Lagrangian Grassmannian} and it will be denoted by
 $\SLGr(3, \PC^3)$. It is also easy to see that it can be represented as
 a homogeneous space
\[
\SLGr(3, \PC^3)= \SLR/\SO,
\]
 see Proposition \ref{Prp:Grass}.
 We now define two bundles over $\CMf$:
\begin{align*}
FL_1^H &= \{ (v,V)\mid  v \in \CMf, \; v \in V, \;
 V  \in  \SLGr(3, \PC^3)\}, \\
FL_2^H &=
 \left\{ (w,\mathcal{W})\;\big|\;
\begin{array}{l}
\text{$w \in \CMf$, $\mathcal W$ is
 a special regular para-complex}  \\
\text{ flag over $w$ in $\PC^3$ satisfying $W_1 = \PC w$}
\end{array}
 \right\}.
\end{align*}
 Moreover, we define
\begin{align*}
 FL_3^H &= \left\{ U P_H^{\epsilon} \;U^T\;\Big|\; \mbox{$U \in \SLR$ and $P
 =\begin{pmatrix}
   0 & \epsilon^2 & 0 \\
   \epsilon^4 &0 & 0 \\
   0 & 0 & -H
  \end{pmatrix}$}\right\},
\end{align*}
 where $\epsilon = e^{\pi i/3}$.
 Then $FL_j^H(j=1, 2, 3)$ are mutually equivariantly diffeomorphic
 $6$-symmetric spaces relative to $\sigma_H$, and they are $7$-dimensional.
\begin{align*}
FL_1^H \cong  FL_2^H \cong FL_3^H = \SLR/\Uone,
\end{align*}
 see Theorem \ref{Thm:3.3}.
 There are natural projections from $\SLR$:
\begin{gather*}
\pi_j: \SLR \to FL_j^H, \quad (j =1, 2, 3).
\end{gather*}

We now further define three spaces:
\[
 Fl_2^H = \{\mathcal W \mid \mbox{$\mathcal W$ is a regular para-complex flag in $\PC^3$} \},
\]
and
\[
\widetilde {Fl_2^H} = \{U (P_H^{\epsilon} (P_H^{\epsilon})^T) U^{-1}\;|\; U \in \SLR\}, \quad
\widetilde \SLGr(3, \PC^3) = \{U (P_H^{\epsilon} (P_H^{\epsilon})^T P_H^{\epsilon}) U^{T}\;|\; U \in \SLR\}.
 \]
 It is easy to see that
 \begin{align*}
\SLGr(3, \PC) &= \SLR/\SO,\quad \widetilde{\SLGr}(3, \PC^3)  = \SLR/\SO,
\end{align*}
 and thus the spaces $\SLGr(3, \PC^3)$ and $\widetilde{\SLGr}(3, \PC^3)$ are
 naturally equivariantly diffeomorphic,
 that is, there exists a diffeomorphism $\phi :\SLGr(3, \PC^3) \to 
 \widetilde{\SLGr}(3, \PC^3)$ such that $\phi (g. x) = g. \phi (p)$ for 
 $g \in \SLR$ and $p \in \SLGr(3, \PC)$.
 symmetric spaces relative to $\sigma_H^3$, and they are $5$-dimensional,
  see Appendix \ref{app:bundles}.

 It is also easy to see that
\begin{align*}
 Fl_2^H &= \SLR/D_3, \quad  \widetilde{Fl_2^H} = \SLR/D_3,
\end{align*}
 and the spaces $Fl_2^H$ and $\widetilde{Fl_2^H}$ are
 naturally equivariantly diffeomorphic
 $3$-symmetric spaces relative to $\sigma_H^2$, and they are $6$-dimensional,
 see again Appendix \ref{app:bundles}.
 There are further natural projections:
\begin{gather*}
\tilde \pi_1: FL_1^H \to \SLGr(3, \PC^3), \quad
\tilde \pi_{3, 1}: FL_2^H \to \widetilde{\SLGr}(3, \PC^3), \quad \\
\tilde \pi_2: FL_2^H \to Fl_2^H, \quad
\tilde \pi_{3, 2}: FL_3^H \to \widetilde{Fl_2^H}.
\end{gather*}
 Schematically, we have the  following diagram:
\begin{equation}\label{eq:diagram}
\begin{diagram}
    \node[3]{\SLR}\arrow{sw,l}{\pi_1} \arrow{s,l}{\pi_3} \arrow{se,l}{\pi_2} \\
    \node[2]{FL_1^H}\arrow{sw,l}{\tilde \pi_1} \node{FL_3^H}\arrow{sw,l}{\tilde\pi_{3,1}} \arrow{se,l}{\tilde\pi_{3,2}}
  \node{FL_2^H}\arrow{se,l}{\tilde \pi_2}  \\
    \node{\SLGr(3, \PC^3)\;\;\cong}   \node{\widetilde{\SLGr}(3, \PC^3)}
  \node[2]{\widetilde{Fl_2^H}\quad \quad \cong} \node{Fl_2^H}\\
\end{diagram}
\end{equation}

\subsection{Ruh-Vilms type theorems  associated with the Gauss maps}
 We will define three Gauss maps taking values in the various bundles given in
 the previous subsection for any liftable immersion  $f : M^2 \rightarrow \DPt$
 with $M^2$ a Riemann surface.

 We assume from now on $M^2 = \D$, and that
 $\f$ is a special lift of $f$. Then we define the frame
 $\mathcal F: \D \to  \GLR$ as in Lemma \ref{det=1} such that
 $\det \mathcal F =1$, that is,
\begin{equation}\label{eq:normalizedframe}
 \mathcal F: \D \to \SLR.
 \end{equation}
 The frame $\mathcal F$ will be called the \textit{normalized frame}.
 Note that the function $\rho$ has been chosen now as in Remark \ref{rm:rho}
 and will generally  not coincide with $\rho_0$ as in Proposition \ref{prp:rho0}.

\begin{definition}\label{dfn:normalizedGauss}
 Let $\mathcal F:\D \to \SLR$ be the normalized frame and
 $\pi_i (i=1, 2, 3)$, $\tilde \pi_i,  \tilde \pi_{3, i} (i=1, 2)$
 be the projections given in \eqref{eq:diagram}.
 Then the maps
 \begin{equation}\label{eq:Gaussmap}
\left\{
\begin{array}{l}
  \displaystyle \g_j = \pi_j \circ \mathcal  F :\D \to FL_j^H \quad (j =1, 2, 3),\\[0.1cm]
\displaystyle {\mathcal H_1} =
 \tilde \pi_{1} \circ \pi_{1} \circ \mathcal  F :\D \to
\SLGr(3, \PC^3),  \\[0.1cm]
\displaystyle {\mathcal H_2} =
 \tilde \pi_{2} \circ \pi_{2} \circ \mathcal  F :\D \to
Fl_2^H,  \\[0.1cm]
 \displaystyle {\mathcal{H}}_{3, 1} = \tilde \pi_{3, 1} \circ \pi_3 \circ \mathcal  F : \D \to \widetilde{\SLGr}(3, \PC^3), \\[0.1cm]
 \displaystyle {\mathcal{H}}_{3, 2} = \tilde \pi_{3, 2} \circ \pi_3 \circ \mathcal  F : \D \to \widetilde{Fl_2^H},
\end{array}
  \right.
\end{equation}
 will be called the \emph{Gauss maps} of $f$.
\end{definition}
 We finally arrive at Ruh-Vilms type theorems, which is
 an exact analogue to Theorem 3.6 in \cite{DoKoMa19}.
\begin{theorem}[Ruh-Vilms theorems for $\sigma_H, \sigma_H^2$ and $\sigma_H^3$]\label{Thm:3.6}
 With the notation used above we consider
any liftable immersion into $\DPt$ and the Gauss maps defined in
 \eqref{eq:Gaussmap}.
 Then the following statements hold$:$
\begin{enumerate}
\setlength{\itemsep}{0cm}
\renewcommand{\labelenumi}{(\arabic{enumi})}
 \item[{\rm (1)}]  $\mathcal{G}_j$ $(j=1, 2, 3)$ is primitive harmonic map into $FL_{j}^H$
 if and only if  $\mathcal F$ is primitive harmonic relative to $\sigma_H$
 if and only if the corresponding surface is a minimal Lagrangian
 immersion into $\DPt$.

\item[{\rm (2)}] $\mathcal H_2$ or $\mathcal H_{3, 2}$ is primitive harmonic  in $Fl_2^H$ or $ \widetilde{Fl_2^H}$
 if and only if $\mathcal F$
is primitive harmonic relative to $\sigma_H^2$ if and only if
the corresponding surface is a minimal immersion into
$\DPt$.

\item[{\rm (3)}]  $\mathcal H_1$ or $\mathcal H_{3, 1}$ is primitive harmonic map
 into $\SLGr(3, \PC^3)$ or $\widetilde{\SLGr}(3, \PC^3)$
 if and only if  $\mathcal F$ is primitive harmonic
 relative to $\sigma_H^3$ if and only if
 the corresponding surface is either a minimal Lagrangian immersion
 or a flat homogeneous immersion  into $\DPt$.
\end{enumerate}

\end{theorem}

\begin{proof}
 The first equivalence in $(1)$ is a consequence of the definition
 of primitive harmonicity into a $k$-symmetric space, and the second equivalence in $(1)$
 has been stated in Theorem \ref{equivprimitive}.
 The proofs for $(2)$ and $(3)$ are similar.
\end{proof}

\section{Appendix}
\appendix
\section{Basic results for $\DP$ and $\CM$}\label{app:basic}

\subsection{The manifold $\CM$}
\subsubsection{Para-complex vector space}\label{subsubsc:para}
 Let us briefly recall the
 $n$-dimensional \textit{para-complex vector space} $(\PC)^n$,
 see \cite{AMT, CFG}:
 First define a commutative and associative multiplication ``$\cdot$'' on $\R^2$  by
 \begin{equation}
  (x, y) \cdot (\tilde x, \tilde y)
 = (x \tilde x + y \tilde y, x \tilde y + y \tilde x),
 \end{equation}
 for $(x, y), (\tilde x, \tilde y) \in \R^2$. Then $(1,0) =: 1$ is the unit element.
 Choosing $\ip = (0, 1)$, $(i^{\prime})^2 = 1$ follows,
 and it is called the \textit{para-imaginary unit}.
 Then the \textit{para-complex numbers} $\PC$ are defined by
\begin{equation}\label{eq:paracomplex}
\PC = \R + \ip \R = \left\{ x + \ip y \;|\; x, y \in \R\right\},
\end{equation}
 and the multiplication on $\PC$ is given by
\begin{equation}
 z w = x \tilde x + y \tilde y + \ip (x \tilde y + y \tilde x),
\quad z = x + \ip y, \quad w = \tilde x + \ip \tilde y.
\end{equation}
 Note that $i^{\prime}$ corresponds to the \textit{para-complex structure} $\Ip$ on $\R^2$ given by
\[
  \Ip : (x, y) \in \R^2 \mapsto (y, x) \in \R^2.
\]
 Then the \textit{para-complex conjugation} $z^*$ of $z= x+ \ip y \in \PC$ is
 defined by $z^* = x - \ip y \in \PC$, the real and imaginary
 parts of $z$ are defined by $x$ and $y$, respectively. Note that $(zw)^* = z^*w^*$.
 The \textit{natural scalar product} of $\PC$ is defined by
 \begin{equation}
  \langle z,  w\rangle_h = z^* w.
 \end{equation}
 This is a \textit{para-Hermitian form}, that is, it is
 $\PC$-antilinear in the first component and $\PC$-linear in
 the second component, respectively, and $\langle w,  z\rangle_h
 = \langle z,  w\rangle_h^*$ holds.
 Note that $\langle z, z \rangle_h$
 takes values in $\R$ and
 $\langle z,  z \rangle_h=0$
 if and only if $z = a \pm \ip a$ for some $a \in \R$.
 The  set of invertible elements will be denoted by
 $(\PC)^{\times}$ and  any element $z \in \PC$ such that
 $\langle z,  z \rangle_h\neq 0$
 is obviously invertible, and we have  $z^{-1} =z^*/ \langle z, z \rangle_h$.

 The $n$-dimensional para-complex vector space $(\PC)^n$ is
 the $n$-fold direct product of $\PC$. The para-Hermitian form of
 $(\PC)^n$ is defined by
\begin{equation}\label{eq:para-inner}
\langle u, v \rangle_h  = u^{*T} v, \quad u, v \in (\PC)^n.
\end{equation}
It is invariant under the transformation
\begin{equation*}
(\PC)^n \ni u \mapsto \left( \frac{1+i^{\prime}}{2}A + \frac{1-i^{\prime}}{2}A^{-T} \right)u \in (\PC)^n,
\quad
 A \in {\rm GL}_n \mathbb R.
\end{equation*}

 It induces the pseudo-Riemannian metric $g$ with signature $(n, n)$
 and the para-Hermitian form $\omega$ on
 $(\PC)^n$ as
\begin{equation}\label{eq:gomega}
 g (u, v) = \Re \langle u, v \rangle_h \quad\mbox{and}\quad
 \omega (u, v) = -\Im \langle u, v \rangle_h,
\end{equation}
 respectively.
 The (extension of the) para-complex structure $\Ip$ (to  $(\PC)^n$)
 is clearly
 parallel with respect to the Levi-Civita connection (relative to $g$),
 and  it is also clear that $g (\ip u, v) = \omega(u, v)$
 and $g(\ip u, \ip v) = - g(u, v)$ hold. Thus
 $(\PC)^n$ is a para-Hermitian manifold.
 Moreover,  $\omega$ is closed and thus it is a symplectic form and
 $(\PC)^n$ is a para-K\"ahler manifold.
 For more details on para-complex manifolds see \cite{CFG}.


\subsubsection{The geometry of $\CM$}
 From Section \ref{subsc:fibration} we  recall the definition  of $\CM$,
\[
 \CM= \left\{ (x,\chi) \in \mathbb R^{n+1} \times \mathbb R_{n+1} \mid \langle x,\chi \rangle = 1 \right\},	
 \]
 where $\R^{n+1}$ denotes the usual $n+1$-dimensional Euclidean space, $\R_{n+1}$ its dual vector space and $\langle \cdot ,\cdot \rangle$ the natural pairing.
 Since the function $h(x, \chi) = \langle x, \chi \rangle - 1$
 has only regular values, we obtain that
 $\CM$ is an embedded submanifold of $\R^{n+1} \times \R_{n+1}$.

 It is straightforward to verify that the tangent space  $T_{(x,\chi)} \CM$ of $\CM$
 can be realized by pairs of vectors
 $(\widehat{\X}, \widehat{\Xt}) \in \R^{n+1} \times  \R_{n+1}$ satisfying
\begin{equation} \label{relationtangent2}
 \langle \widehat{\X}, \chi \rangle + \langle x, \widehat{\Xt} \rangle = 0.
\end{equation}
 To rephrase $\CM$ by the para complex number $(z)$,
 we now introduce the natural standard inner product on $\R^{n+1}$,
 i.e. $\langle x, y\rangle= \sum_{i=1}^{n+1} x_i y_i$ for
 $x = \sum_{i=1}^{n+1} x_i e_i, y = \sum_{i=1}^{n+1} y_i e_i \in \R^{n+1}$,
  and using it we identify $\R^{n+1}$ with
 $\R_{n+1}$. This identification then also induces the standard inner product on $\R_{n+1}$. As a consequence of this we will replace ``$\R_{n+1}$''
 by $\R^{n+1}$ from here on.  But to indicate to which space a vector belongs we will continue to use latin letters $x,y,\dots$
 for the first component and Greek letters for the second component.

 Then the real coordinates $(x, \chi)$ and the para-complex coordinates $(z)$
 can be identified by
\begin{equation}\label{eq:correspondence}
(x, \chi) \Longleftrightarrow z = \frac{x+ \chi}{2} + \ip \frac{x- \chi}{2}
 = \frac{1+ \ip }{2} x + \frac{1- \ip }{2} \chi.
\end{equation}
 It is easy to see that $\langle x, \chi\rangle = \langle z, z\rangle_h$ under
 the above identification.
 More generally, the  para-Hermitian form \eqref{eq:para-inner}
 for the vectors  $z = \frac12 (x+ \chi + \ip (x- \chi))$ and
 $w = \frac12(\tilde x+ \tilde \chi + \ip (\tilde x- \tilde \chi))$
 can be computed
 as
\[
 \langle z, w \rangle_h = \frac12 \left\{
\langle x,\tilde \chi  \rangle+\langle \tilde x,\chi  \rangle + \ip \left(- \langle x,\tilde \chi  \rangle+\langle \tilde x, \chi  \rangle
\right)\right\}.
\]
 Using
 the para-complex numbers $\PC$ for the description of  $\CM$ we obtain
\[
 \CM= \left\{ z \in (\PC)^{n+1} \mid \langle z, z \rangle_h =1 \right\}
\]
 and  the tangent space  $T_z \CM$ to $\CM$  at the point
 $z \in  (\PC)^{n+1}$  can be realized by the vectors $\hat w \in (\PC)^{n+1}$
 satisfying
\begin{equation} \label{relationtangent}
\Re \langle \hat w, z \rangle_h = 0.
 \end{equation}
\begin{prop}
Retaining the notation introduced above we have the following$:$
\begin{enumerate}
\setlength{\itemsep}{0cm}
\renewcommand{\labelenumi}{(\arabic{enumi})}
\item[\rm (1)]  The set  $\CM$ is an embedded $2n+1$-dimensional submanifold of $(\PC)^{n+1}=
\R^{2n+2}$.
\item[\rm (2)]  The manifold $\CM$
 is a contact manifold  defined by the contact $1$-form on $\CM:$
\begin{equation}\label{eq:contactform}
 \zeta (z) = \Im \langle z, \, \rangle_h.
\end{equation}

\item[\rm (3)] The Reeb vector field $\mathcal{R}$ for the contact manifold $\CM$ with contact $1$-form $\zeta$ is given by
\[
\mathcal{R}(z) = \ip z.
\]										
\end{enumerate}
\end{prop}
We also have the following properties on $\CM$.
\begin{lemma}
 Let $\hat g$ be the pseudo-Riemannian metric on $\CM$ induced from
 $g$ defined by \eqref{eq:gomega} on $(\PC)^{n+1}$. Then the following statements hold$:$
\mbox{}
\begin{enumerate}
\setlength{\itemsep}{0cm}
\renewcommand{\labelenumi}{(\arabic{enumi})}
 \item[\rm (1)] For two tangent vectors $\hat v, \hat w\in T_{z} \CM$,  we have
\begin{equation}
\hat{g}(\hat v, \hat w ) = \Re \langle \hat v, \hat w\rangle_h\quad.
\end{equation}
In real coordinates $(x, \chi)$, for tangent vectors
$(\widehat{\X}, \widehat{\Xt}), (\widehat{\Y}, \widehat{\Yt}) \in T_{(x, \chi)} \CM$
\begin{equation}\label{eq:hatg}
 \hat{g}(\, (\widehat{\X}, \widehat{\Xt}), (\widehat{\Y}, \widehat{\Yt}) \,)\
 = \frac12 (\, \langle \widehat{\X}, \widehat{\Yt} \rangle + \langle \widehat{\Y},
 \widehat{\Xt} \rangle\,).
\end{equation}
 \item[\rm (2)] For an arbitrary vector  $\hat w \in T_{z} \CM$,
 the para-complex structure $\hat \Ip$ of $(\PC)^{n+1}$ naturally acts on $\hat w$ as
\begin{equation}
\hat \Ip \hat w = \ip \hat w.
\end{equation}
 For a real tangent vector $(\widehat{\X}, \widehat{\Xt}) \in T_{(x, \chi)}\CM$,
 $\hat \Ip$ acts as
\begin{equation}
\hat \Ip (\widehat{\X}, \widehat{\Xt}) = (\widehat{\X}, - \widehat{\Xt}).
\end{equation}
\end{enumerate}
\end{lemma}
\subsection{The space $\DP$ and the fibration $\piH: \CM \rightarrow \DP$}

\subsubsection{The basic fibration}
We recall the para-K\"ahler complex projective space
\[
 \DP = \{ ([x],[\chi]) \in \mathbb {R}P^n \times \mathbb {R}P_n \mid
\langle x, \chi\rangle>0 \},
\]
defined in \eqref{eq:DP}\footnote{The original definition is $[x] \not\perp [\chi]$},
 where $\mathbb {R}P^n$ is the $n$-dimensional real projective space,
$\mathbb {R}P_n$ is the projective space of the dual space $\R_{n+1}$ of $\R^{n+1}$,
sometimes also called the ``dual real projective space''.
Note that the equivalence class of an element in $\DP$ can be defined as follows:
 Set
\[
 (\mathbb{R}^{n+1} \times \mathbb{R}_{n+1})^+=  \left\{
 (x, \chi)\mid \langle x, \chi \rangle >0 \right\}.
\]
 Then the quotient  is defined as follows:
 $(\mathbb{R}^{n+1} \times \mathbb{R}_{n+1})^+ \ni (x, \chi) \sim
 (\tilde x,\tilde \chi) \in (\mathbb{R}^{n+1} \times \mathbb{R}_{n+1})^+$  iff
 there exists a pair $(p, q) \in \R_{+} \times \R_{+}$ or
 $(p, q) \in \R_{-} \times \R_{-}$
such that $(\tilde x, \tilde \chi)= (p x, q \chi)$. Thus
\[
 \DP = (\mathbb{R}^{n+1} \times \mathbb{R}_{n+1})^+/\sim.
\]
By using the correspondence between real and para-complex coordinates
  as stated in  \eqref{eq:correspondence}, we can write
 $\DP$ in the form
\begin{equation}\label{eq:DP2}
\DP= \left\{ [z] \mid  z \in (\PC)^{n+1}, \langle z, z\rangle_h >0 \right\},
\end{equation}
 where $[z]$ denotes an equivalence class defined as follows. Two elements $w,z \in (\PC)^{n+1}$ satisfying $\langle z, z\rangle_h>0$,
$\langle w, w\rangle_h>0$ are equivalent, $z\sim w$, if and only if there exists $\lambda \in \PC$ such that $z = \lambda w$. Note that then automatically $\langle \lambda, \lambda\rangle_h>0$.
It is easy to verify that the elements $[z] \in \DP$ in \eqref{eq:DP2} and
 $([x],[\chi]) \in \DP$ in the original definition in \eqref{eq:DP}, respectively, are bijectively related by the equivalence:
\[
 ([x],[\chi])  \Longleftrightarrow
 [z]=\left[\frac{x + \chi}{2} + \ip \frac{x - \chi}{2}\right].
\]
In fact, if $(\tilde x, \tilde \chi)
 = (p x, q \chi) \in (\R^{n+1} \times \R_{n+1})^+$
 for some $(p, q) \in \R^{\pm} \times \R^{\pm}$, then
  $\tilde z = \lambda z$ by choosing
\begin{equation}\label{eq:correpondenceR}
 \lambda = \frac{p+q}{2} + \ip \frac{p-q}{2} = \frac{1+\ip}{2}p + \frac{1-\ip}{2}q.
\end{equation}
 Vice versa, decomposing $\lambda \in \PC$ satisfying $\langle \lambda, \lambda\rangle_h>0$ as in \eqref{eq:correpondenceR}, we get $(p,q) \in \R^{\pm} \times \R^{\pm}$, and $\tilde z = \lambda z$ implies $(\tilde x, \tilde \chi) = (p x, q \chi)$.
Finally, from Section \ref{subsc:fibration} we recall the projection map
 $\piH : \CM \rightarrow \DP, \quad (x,\chi) \mapsto ([x], [\chi])$.
By using the para-complex coordinates $(z)$ and the correspondence in
 \eqref{eq:correspondence}, we easily see that
\begin{equation}\label{eq:Hopffibration}
 \piH : \CM \rightarrow \DP, \quad z \mapsto [z]
\end{equation}
 holds.

\subsubsection{The horizontal and vertical distributions}
Recall that the tangent space of $\CM$ at $z$ is given by
 $T_z \CM= \left\{\hat w \in (\PC)^{n+1} \mid \Re \langle \hat w, z\rangle_h=0\right\}$.
 Since the fiber at $[z] \in \DP$ of the fibration $\piH : \CM \rightarrow \DP$ is
\[
     \{ e^{\ip t} z \in \CM \mid t \in \R \},
\]
where $e^{\ip t} = \cosh t + \ip \sinh t$ and
 $\langle e^{\ip t}, e^{\ip t}\rangle_h =1$,
 the kernel of $d \piH$ is the direction of the Reeb vector field $\mathcal R(z)$,
 that is $\ip z$.
Therefore it is natural to decompose a tangent vector $\hat w \in T_z \CM$ as
\begin{equation}
\hat w = (\hat w- \langle \hat w, z\rangle_h z) +\langle \hat w, z\rangle_h z,
\end{equation}
where $\langle \hat w, z\rangle_h$ takes imaginary values since
 $\Re \langle \hat w, z\rangle_h=0$, and
 $\langle \hat w- \langle \hat w, z\rangle_h z,z \rangle_h = 0$.
We define the \textit{horizontal distribution} by
\begin{align*}
\widehat{\mathcal{H}}_z  =\{ \hat w \in  (\PC)^{n+1} \mid \langle \hat w, z \rangle_h = 0 \},
\end{align*}
and on the other hand the \textit{vertical distribution} by
\begin{equation} \label{vertical}\\
 \quad\widehat{\mathcal{V}}_z =  \R (\ip z).
\end{equation}
Thus we have $\hat w \in T_z \CM = \widehat{\mathcal H}_z \oplus \widehat{\mathcal{V}}_z$.

In particular, the vertical distribution is integrable (since it is $1$-dimensional),
 while the horizontal distribution is not integrable.
 In real coordinates $(x, \chi)$ the horizontal and the vertical distribution
 in the tangent space $T_{(x, \chi)} \CM$ can be given by
\begin{align*}
\widehat{\mathcal{H}}_{(x,\chi)}&= \{ (\widehat{\X}, \widehat{\Xt}) \in  \R^{n+1} \times  \R_{n+1} \mid \mbox{$\langle \widehat{\X}, \chi \rangle = 0$ and
$\langle x, \widehat{\Xt} \rangle = 0$}  \},
 \\ 
\widehat{\mathcal{V}}_{(x,\chi)} &= \R (x,-\chi),
\end{align*}
 which are the same equations as in Section \ref{subsc:fibration}.

Finally we consider a finer decomposition of the tangent space of $\CM$:
In view of \eqref{relationtangent} one verifies easily
that the horizontal distribution can be decomposed in the form
\begin{equation*}
\widehat{\mathcal{H}}_{z} =  \widehat{\mathcal{H}}^+_{z} \oplus  \widehat{\mathcal{H}}^-_{z},
\end{equation*}
where
 $\widehat{\mathcal{H}}^{\pm}_{z} =  \{ \hat w \pm \ip \hat w \mid \hat w \in \widehat{\mathcal{H}}_{z}\}$.
 In fact, for every $\hat w \in T_z\CM$, that is $\Re \langle z, \hat w\rangle_h=0$,
 we have the decomposition
\begin{align} \label{explicitdec}
\hat w & = (\hat w- \langle \hat w, z\rangle_h z) +\langle \hat w, z\rangle_h z \in \widehat{\mathcal H}_z
 \oplus \widehat{\mathcal V}_z \\
 & =(\hat w- \langle \hat w, z\rangle_h z)_+ +\langle \hat w, z\rangle_h z +
 (\hat w- \langle \hat w, z\rangle_h z)_-\in \widehat{\mathcal H}_z^+
 \oplus \widehat{\mathcal V}_z \oplus \widehat{\mathcal H}_z^-,
\nonumber
\end{align}
where $(\hat w- \langle \hat w, z\rangle_h z)_{\pm}= \frac{1\pm\ip}{2}(\hat w- \langle \hat w, z\rangle_h z)$.

In real coordinates $(x, \chi)$, any tangent vector $(\widehat{\X},  \widehat{\Xt}) \in T_{(x, \chi)}\CM$, that is, $\langle x,  \widehat{\Xt} \rangle
 = - \langle \widehat{\X}, \chi\rangle$, can be decomposed as
\begin{align} \label{explicitdec:real}
(\widehat{\X}, \widehat{\Xt}) & =
(\widehat{\X} -\langle \widehat{\X}, \chi \rangle x,
\widehat{\Xt} - \langle x,  \widehat{\Xt} \rangle \chi )
+ \langle \widehat{\X}, \chi \rangle (x, -\chi )\in \widehat{\mathcal H}_{(x, \chi)}
 \oplus \widehat{\mathcal V}_{(x, \chi)}, \\
 &=
 (\widehat{\X} -\langle \widehat{\X}, \chi \rangle x, 0)
+ \langle \widehat{\X}, \chi \rangle (x, -\chi )
+
(0,\widehat{\Xt} - \langle x,  \widehat{\Xt} \rangle \chi)
\in  \widehat{\mathcal H}_{(x, \chi)}^+
 \oplus \widehat{\mathcal V}_{(x, \chi)} \oplus \widehat{\mathcal H}_{(x, \chi)}^-.
\nonumber
\end{align}
 Let us summarize the above discussion as the following proposition.
 \begin{prop} \label{prop:A3}
 Retaining the notation introduced above we obtain:
 \begin{enumerate}
\item[\rm (1)] The projection map $\piH : \CM \rightarrow \DP, \quad (x,\chi) \mapsto ([x], [\chi])$  is a
pseudo-Riemannian submersion with vertical distribution  $\widehat{\mathcal{V}}_z =  \R (\ip z)$ and horizontal distribution $\widehat{\mathcal{H}}_z ,$ $z \in (\PC)^{n+1}$. Moreover, the horizontal distribution has the following property:
\[
\mbox{For $\hat w \in T_{z} \CM$ we have $\hat w \in \widehat{\mathcal{H}}_{z}$
if and only if $\hat\Ip \hat w  \in \widehat{\mathcal{H}}_{z}$.}
\]

  \item[\rm (2)]  There are natural decompositions for $T_{z} \CM$
\begin{align} \label{tangent_decomp}
T_{z} \CM = \widehat{\mathcal{H}}_{z} \oplus   \widehat{\mathcal{V}}_{z}
=\widehat{\mathcal{H}}^+_{z} \oplus   \widehat{\mathcal{V}}_{z}    \oplus \widehat{\mathcal{H}}^-_{z}.
\end{align}
 The distributions  $\widehat{\mathcal{H}}^{\pm}$ are $n$-dimensional
 distributions in $T \CM$.
The subspaces $\widehat{\mathcal{H}}^{\pm}_{z}$
 are eigenspaces of $\hat \Ip$ with eigenvalues $1$ and $-1$, respectively.
 Moreover, $\widehat{\mathcal{H}}^{\pm}_{z}$ are isotropic and orthogonal to the vertical subspace $\widehat{\mathcal{V}}_{z}$ with respect to the symmetric bilinear form $\hat{g}$.
The restriction of $\hat g$ to $\widehat{\mathcal{H}}_{z} =  \widehat{\mathcal{H}}^+_{z} \oplus  \widehat{\mathcal{H}}^-_{z}$
is non-degenerate with signature $(n,n)$.

\item[\rm (3)] The distributions
$ \widehat{\mathcal{H}}^+_{z} ,  \widehat{\mathcal{V}}_{z} $, and $ \widehat{\mathcal{H}}^-_{z} $ are integrable.

\end{enumerate}
\end{prop}
\begin{proof}
 $(1)$: We need to introduce a pseudo-Riemannian metric on $\DP$.
 The manifold $\CM$ has a pseudo-Riemannian metric with signature $(n, n+1)$
 induced from a natural para-Hermitian
 form on $(\PC)^{n+1}$ of constant curvature. Then it is straightforward to
 introduce a pseudo-Riemannian metric on $\DP$ through the fibration $\piH$.
 Namely, it is induced from the horizontal
 part $\widehat{\mathcal{H}}_{z}$ of  $T_z\CM$.
 In fact, the pseudo-Riemannian metric on $\DP$ is an indefinite version of
 the Fubini-Study metric of the projective space $\mathbb C P^n$.

 Let $z,\tilde z \in \CM$ be such that $\tilde z = \lambda z$ for some $\lambda \in \PC$, and let $\hat w \in T_z\CM$. Then $d\piH$ maps $\lambda\hat w \in T_{\tilde z}\CM$ to the same vector in $T\DP$ as $\hat w$. Moreover, for arbitrary $\hat v,\hat w \in T_z\CM$ we have $\hat g_z(\hat v,\hat w) = \hat g_{\tilde z}(\lambda\hat v,\lambda\hat w)$, because $\langle \lambda,\lambda \rangle_h = 1$. The orthogonal complement $(\operatorname{ker} d\piH)^{\perp}$ is
 exactly the horizontal distribution $\widehat{\mathcal H}$, since for
 $\hat v = \ip z\in \operatorname{ker} d\piH = \widehat{\mathcal V}_z$  and
 $\hat w \in \widehat{\mathcal H}_z$,
\[
 \hat g_z (\hat v, \hat w)= \Re \langle \hat v, \hat w \rangle_h
 = - \Im \langle z, \hat w \rangle_h = 0.
\]
 Therefore
 $d \piH :(\operatorname{ker} d\piH)^{\perp} \to T\DP$ is
 an isometry and $\piH$ is a pseudo-Riemannian submersion.
 The second statement is a straightforward computation.

 $(2)$: The second statement $(2)$ follows from \eqref{explicitdec} and from the orthogonality relation deduced above.

 $(3)$:  For $ \widehat{\mathcal{V}}_{z} $  this is clear,
 since the distribution is $1$-dimensional.
For $ \widehat{\mathcal{H}}^{\pm}_{z}$ we obtain as integral manifolds
 $\mathcal{I}_{z}^{\pm}$ through $z$
the sets
\begin{equation}
\mathcal{I}_{z}^{\pm} = \left\{ p \in \CM \mid z \mp \ip z = p \mp \ip p \right\}.
\end{equation}
 In real coordinates $(x, \chi)$,
\begin{equation}
\mathcal{I}_{(x, \chi)}^{+} = \left\{ (y, \pi) \in \CM \mid \pi = \chi \right\},
 \quad
\mathcal{I}_{(x, \chi)}^{-} = \left\{ (y, \pi) \in \CM \mid y = x \right\}.
\end{equation}
This completes the proof.
\end{proof}

On the other hand, the contact structure on $\CM$ induces
a natural symplectic structure on the horizontal distribution $\widehat{\mathcal H}$.
\begin{prop}
\mbox{}
\begin{enumerate}
\setlength{\itemsep}{0cm}
\renewcommand{\labelenumi}{(\arabic{enumi})}
\item[\rm (1)] The differential of the contact form $\zeta$ in \eqref{eq:contactform} defines a symplectic $2$-form on the horizontal distribution $\widehat{\mathcal H}:$
\[
\hat  \omega_z (\hat v, \hat w) = d \zeta (\hat v, \hat w)  = - \Im \langle \hat v, \hat w\rangle_h,
\quad \hat v, \hat w \in \widehat{\mathcal H}_z
\]
 For real horizontal vectors $(\widehat{\X},\widehat{\Xt}),
 (\widehat{\Y},\widehat{\Yt}) \in \widehat{\mathcal H}_{(x, \chi)}$,
\begin{align}\label{eq:omegahat}
\hat{\omega}_{(x,\chi)}(\,(\widehat{\X},\widehat{\Xt}),
 (\widehat{\Y},\widehat{\Yt})\,) = \frac12
(\,\langle \widehat{\X},\widehat{\Yt} \rangle - \langle \widehat{\Y},
 \widehat{\Xt} \rangle\,).
\end{align}
The form $\hat\omega$ is closed and its kernel is given by the vertical
 subspace $\widehat{\mathcal{V}}_{z}$.

\item[\rm (2)] The covariant derivative with respect to the Levi-Civita connection of $\hat g$ of the tensor $\hat \Ip$, and hence also the form $\hat\omega$,
 in the direction of the vertical subspace vanishes. The same holds for the vector field $\ip z $ generating the vertical subspace.
\end{enumerate}

\end{prop}

\begin{proof}
The proposition is proven by direct calculation.
\end{proof}

\subsection{Homogeneous structures of $\CM$ and $\DP$}
\subsubsection{The action of $\SLRn$ on $\CM$}
It is clear that the group $\SLRn$  acts naturally by matrix multiplication on
 $\R^{n+1}$. Then the contragredient representation of  $\SLRn$
is given by $g \in  \SLRn,
\chi \in \R_{n+1} , u \in \R^{n+1}$:
\begin{equation} \label{contragredient}
(g^*\chi)(u) = \chi( g^{-1}u).
\end{equation}
 Using these definitions we obtain the following.
\begin{prop}
 Denoting by $e_1, \dots e_{n+1}$ the natural basis of $\R^{n+1}$ and by $\delta_j$, given by $\delta_j(e_k) = \delta_{jk},$ its dual basis, we have the following$:$
\mbox{}
\begin{enumerate}
\setlength{\itemsep}{0cm}
\renewcommand{\labelenumi}{(\arabic{enumi})}
 \item[\rm (1)]
The manifold $\CM$ is a connected homogeneous space under the action of $\SLRn$
given by
\begin{equation}
g(u,\chi) = (g u, g^*\chi),
\end{equation}
where $u,\chi, g$ are as above.
 \item [\rm (2)] The isotropy group at $(e_1, \delta_1)$ of the action just stated
 is isomorphic to $\SLRnm \cong \{ 1 \} \times \SLRnm$
 and thus $\CM$ can be written as
 the  homogeneous space
\begin{equation} \label{CM}
 \CM = \SLRn/\mathrm{SL}_n \R.
\end{equation}
\item[\rm (3)]  The group $\SLRn$ acts on $\CM$ by isometries
 and leaves the horizontal distributions   $ \widehat{\mathcal{H}}^{\pm}$ and the vertical distribution invariant.
\end{enumerate}
\end{prop}
\begin{proof}
(1):
 First we note that (\ref{contragredient}) implies that  $\SLRn$  leaves $\CM$ invariant.
 Let $(u, \chi) \in \CM$, then one can take an element $g \in \SLRn$ such that
 $g u=e_1$. Then one can take another $g$ such that
 $ge_1 =e_1$ and $g^* \chi = \delta_{1}$. Thus the action is transitive.

(2):   By the action, the pair $(e_1,\delta_1)$ is stabilized exactly
 by $\diag(1,S)$ with $S \in {\rm SL}_{n} \mathbb R$.
 Therefore the stabilizer is $\mathrm{SL}_n \R \cong \{1\}\times \mathrm{SL}_n \R$.

(3): Since the action of $\SLRn$ is linear, it acts on the tangent vectors
 by the same formulas as on $\CM$. Thus the claim follows.
 \end{proof}

Recall that the actions of
 $ \SLRn$ on $\CM$, on  $(\PC)^{n+1}$  and all other geometric objects
 investigated in  this paper, are induced naturally from the basic matrix action on
 $\R^{n+1}$ and the ``diagonal action'' $ (g, g^*)$, where $g^*$ denotes the contragredient action, see \eqref{contragredient}.
 Moreover, all natural isomorphisms/diffeomorphisms/isometries occurring in this paper are trivially equivariant relative to the corresponding actions of   $ \SLRn$.
 Further, wherever applicable, all these
 actions of  $ \SLRn$ commute.
\begin{remark}
 The center $\mathcal{C}_{n+1}$ of $\SLRn$ is  $\pm I$, if $n$ is even and it is $I$, if $n$ is odd.
 Moreover, $\mathcal{C}_{n+1}$ acts freely and properly on $\CM$.
\end{remark}
 To discuss the geometry of $\DP$, we will use what was already discussed for
 $\CM$  in the previous sections. As an application of these remarks we show:
\begin{prop}
Retaining the notation introduced so far, the following statements hold:
\begin{enumerate}
\setlength{\itemsep}{0cm}
\renewcommand{\labelenumi}{(\arabic{enumi})}
\item[\rm (1)]  The projection map $\piH : \CM \rightarrow \DP, \quad (x,\chi) \mapsto ([x], [\chi])$,  is
equivariant relative to the natural action of $\SLRn$ on $\CM$ and the natural action of
$\mathrm{PSL}_{n+1} \R \cong \SLRn /\mathcal{C}_{n+1}$ on $\DP$.
\item[\rm (2)]  Similar to equation \eqref{CM} one can represent $\DP$ in the form
\begin{equation}
\DP \cong \SLRn / { \bigcup_{a \in \R^\times}  \left( \{a \} \times (a^*)^{-1}\mathrm{SL}_n \R \right) }
\end{equation}
\item[\rm (2)] The following diagram commutes$:$
\begin{center}
\begin{tikzcd}[column sep=4em,row sep=4em]
{ \{1 \} \times \mathrm{SL}_n \R}  \ar[hookrightarrow]{r}{incl}   \ar[hookrightarrow]{d}[swap]{incl}   &
{ \bigcup_{a \in \R^\times}  \left( \{a \} \times (a^*)^{-1}\mathrm{SL}_n \R \right) }  \ar[hookrightarrow]{d}{incl} \\
 \SLRn  \ar{r}{id}  \ar{d}[swap]{proj}   &    \SLRn   \ar{d}{proj}    \ar{r}{proj} & {\mathrm{PSL}_{n+1} \R}
 \ar{d}{proj} \\
 \CM  \ar{r}{\piH}   &   \DP \ar{r}{id}  &\DP
\end{tikzcd}
\end{center}
 Moreover, the fiber of  $(\lbrack x \rbrack, \lbrack \chi \rbrack)$ under
$\pi_\mathcal{H}$ is $(\pi_\mathcal{H})^{-1}(\lbrack x \rbrack, \lbrack \chi \rbrack) =
\{(ax, a^{-1} \xi)\mid a \in \R^\times \}.$
\end{enumerate}
\end{prop}


\subsubsection{The induced geometry on $\DP$}\label{subsub:induced}
The discussion about the geometry of $\CM$ induces quite directly the basic objects of the geometry of $\DP$. We collect these results in the following theorem.
\begin{theorem}\label{thm:Basics}
We retain the assumptions and the notation of the previous subsections. Then we obtain:
\begin{enumerate}
\setlength{\itemsep}{0cm}
\renewcommand{\labelenumi}{(\arabic{enumi})}
\item[\rm (1)]The differential $d \piH$  induces an isomorphism from the distributions
 $\widehat{\mathcal{H}}^+$ and
 $\widehat{\mathcal{H}}^-$ to the image distributions $\mathcal{H}^+ $
 and $\mathcal{H}^- $, where
\[
\mathcal{H}^{\pm}_{z} =  \{ d\piH (\hat w \pm \ip \hat w) \mid \hat w \in T_z \CM \}.
\]
 The distributions  $\mathcal{H}^+ $ and $\mathcal{H}^- $ are integrable.

\item[\rm (2)] The projection $\piH$ induces naturally a non-degenerate pseudo-Riemannian metric
$g$ on $\DP$ with signature $(n,n)$ such that for all
$v, w\in T_{[z]} \DP$ with
 $\hat v, \hat w \in  \widehat{\mathcal H}_z  \subset T_z \CM$
 and $(d\piH (\hat v), d\piH (\hat w))=(v, w):$
\[
 g(v, w) = \hat{g}(\hat v, \hat w).
\]

\item[\rm (3)]  There exists a para-complex structure $\Ip$ on $\DP$  satisfying
\[
 d\piH \circ  \hat \Ip = \Ip \circ d\piH.
\]
It has the spaces  $\mathcal{H}^+ $ and $\mathcal{H}^- $ as eigenspaces with eigenvalues $+1$ and $-1$ respectively.

\item[\rm (4)] The projection $\piH$ induces naturally a (non-degenerate, closed) symplectic form
$\omega$ on $\DP$ such that
for all
$v, w\in T_{[z]} \DP$ with
 $\hat v, \hat w \in  \widehat{\mathcal H}_z  \subset T_z \CM$
 and $(d\piH (\hat v), d\piH (\hat w))=(v, w):$
\[
  \omega (v, w) = \hat{\omega}(\hat v, \hat w).
\]
 In particular, we have  for all $v, w \in T_{[z]} \DP$
\[
 \omega(v, w) = g (\Ip v, w)\quad \mbox{and}\quad
 g(v, w) = \omega (\Ip v, w).
\]
\item[\rm (5)] 
 The tensor $\Ip$, and hence also $\omega$, are parallel with respect to the Levi-Civita
 connection of $g$, and thus $\DP$ is a para-K\"ahler manifold of dimension $2n$.
\end{enumerate}
\end{theorem}

\emph{Proof of {\rm Proposition \ref{J_horizontal_proposition}:}}
The first statement is just (1) in Proposition \ref{prop:A3}.
For the second statement, note that the preimages of the tangent vectors $X, Y \in
 T_{([x],[\chi])} \DP$ are $(\widehat{\X}, \widehat{\Xt}),
 (\widehat{\Y}, \widehat{\Yt}) \in \widehat{\mathcal{H}}_{(x, \chi)}$,
 respectively.
Then from (4) in Theorem \ref{thm:Basics}, \eqref{eq:hatg}, and \eqref{eq:omegahat} we get
\begin{align*}
(g+\omega)(X,Y)  &= \hat g((\widehat{\X}, \widehat{\Xt}),
 (\widehat{\Y}, \widehat{\Yt})) +
\hat \omega ((\widehat{\X}, \widehat{\Xt}),
 (\widehat{\Y}, \widehat{\Yt})) \\ &=
\langle \widehat{\X},\widehat{\Yt} \rangle.
\end{align*}
 This completes the proof.

\subsection{Immersions into $\DP$ and lifts to $\CM$}

\subsubsection{Immersions and lifts}
 In this paper we investigate immersions $f: M^n \to \DP$  via immersions $\f: M^n \to \CM$.
To make this precise we need
\begin{definition}
Let $f: M^n \to \DP$ be any immersion. Then
\begin{enumerate}
\item[\rm (1)] A smooth map $\f: M^n \to \CM$  is called a lift of $f$ iff $f = \pi_\mathcal{H} \circ \f$.
\item[\rm (2)] If $U\subset M^n$ is an open subset of $M^n$, then a lift of $f_{|U}$ is called a ``local lift''
(of $f$ with respect to $U$).
\end{enumerate}
\end{definition}
 It is easy to see that a  lift is unique up to ``scalings''  of
 the form $(x,\chi) \mapsto (c x, c^{-1}\chi)$ for never vanishing scalar
 functions $c$.

\begin{theorem} \label{thm:A11}
If $M^n$ is a connected, simply connected manifold and $f: M^n \to \DP$ is an immersion, then there exists a lift
$\f: M^n \rightarrow \CM$, i.e. satisfying $ f = \piH \circ  \f$.
\end{theorem}
\begin{proof}
We shall show that the projection $\piH$ can be restricted to a double covering of $\DP$. Introduce arbitrary Euclidean norms in the spaces $\R^{n+1}$, $\R_{n+1}$. Consider the set
\begin{equation} \label{def:S2}
S_{2n} = \{ (x,\chi) \in \CM \mid \langle x, x\rangle  = \langle \chi, \chi\rangle \}.
\end{equation}
Clearly $S_{2n}$ is a smooth manifold. For every point $([x],[\chi]) \in \DP$ there exist exactly two representatives $(x,\chi) \in S_{2n}$, related by a sign change in both components. Thus the restriction of $\piH$ to $S_{2n}$ is a double covering of $\DP$.

Now $M^n$ is simply connected, path connected, and locally path connected. By the unique homotopy lifting property there exists a smooth  lift (in fact, exactly two of them) of $f$ to $S_{2n}$, and hence to $\CM$.
\end{proof}

Before going on we formalize the obtained results about $S_{2n}$.
\begin{prop} \label{aboutS2n}
For the subset $S_{2n}$ defined in \eqref{def:S2} the following statements hold:
\begin{enumerate}
\item The set $S_{2n}$ is an embedded submanifold of $\CM$.
\item $S_{2n}$ is a double cover of $\DP$ under the restriction of $\pi_\mathcal{H}$ to $S_{2n}$.
\end{enumerate}
\end{prop}

\begin{corollary}  \label{aboutlifts} With the notation above we have:
\begin{enumerate}
\setlength{\itemsep}{0cm}
\renewcommand{\labelenumi}{(\arabic{enumi})}
\item[{\rm (1)}] If $f:  M^n \rightarrow \DP$ is an immersion and $U$ any
 simply connected subset of $M^n$, then $f|_U : U \rightarrow \DP$ has a global lift.
 In other words, each immersion $f:  M^n \rightarrow \DP$ has local lifts around each point of $M^n$.

\item[{\rm (2)}] Let
 $f:  M^n \rightarrow \DP$ be  an immersion, $(U_\beta)_{\beta \in \mathcal{J}}$  an open covering of $M^n$ by simply connected charts,
and $\f_\beta : U_\beta \rightarrow \CM$ a local lift of $f$ on $U_\beta$.
Then on the intersection $U_\beta \cap U_\gamma$ of two such charts the  lifts $\f_\beta$
and $\f_\gamma$ are transformed into each other by a
scaling by a uniquely determined never vanishing function
$c_{\beta \gamma} : U_\beta \cap U_\gamma \rightarrow \R \setminus \{ 0 \}$.
These functions form a cocycle
with values in $\R \setminus \{0\} $, i.e. they satisfy $c_{\alpha \beta} c_{\beta \gamma} = c_{\alpha \gamma}$. Moreover,  the
triviality of this cocycle is equivalent to the existence of a global lift of $f$.

\item[{\rm (3)}]  If $f:  M^n \rightarrow \DP$ is  an immersion, then
  there exists a two-fold cover $\hat{M}^n$ of $M^n$ such that the natural lift
  $\hat{f}$ of $f$ has a global
 lift $\hat{\f} :\hat{M}^n \rightarrow \CM$.
\end{enumerate}
\end{corollary}

\begin{proof}
(1) and (2) follow from Theorem \ref{thm:A11}.
Let us prove (3). Recall from Proposition \ref{aboutS2n} that
$S_{2n} \subset \CM$ is a two-fold cover of $\DP$. Pulling back this covering $p$ along $f$ we obtain the commuting diagram
\begin{center}
\begin{tikzcd}[column sep=4em,row sep=4em]
 \hat{M}^n  \ar{r}{\hat{\f}}  \ar{d}[swap]{\hat{p}}   &    S_{2n} \subset \CM  \ar{d}{p}  \\
 M^n  \ar{r}{f}   &   \DP .
\end{tikzcd}
\end{center}

 Thus $f \circ \hat{p}: \hat{M}^n \rightarrow \DP$ has a global lift into $\CM$.
\end{proof}
In general, a given $f: M^n \to \DP$  cannot be lifted :
\begin{prop}\label{prp:nonlift}
Let $f: M^n \to \DP$  be the injective immersion from $M^n = \R P^n$ to $\DP$,
given by $f([x]) = ([x], [x])$. Then $f$ does not have any lift $\f: M^n \rightarrow \CM$.
\end{prop}
\begin{proof}
Assume there exists a lift $\f: \R P^n \rightarrow \CM$ of $f$. Since $f$ is injective,  it is easy to verify that $\f$ is injective. Therefore, since $\R P^n$ is compact, $\f$ actually is an embedding.
Let $\pi_1$ denote the natural projection $\pi_1: \CM \rightarrow \R^{n+1}$.
It is easy to show that  the map  $ \pi_1 \circ \f : \R P^n \rightarrow \R^{n+1}$
also is an embedding. Thus $f$ would define a compact non-orientable submanifold of $\R^{n+1}.$
 But this is a contradiction, since  there does not exist any closed non-orientable submanifold of $\R^{n+1}$ by
see, e.g. \cite{Samelson}.
This contradiction proves the claim.
\end{proof}

As a consequence of the last result we restrict our consideration in this paper to liftable immersions $f: M^n \to \DP$.
\subsubsection{Horizontal lifts}
In other cases, like minimal Lagrangian surfaces in $\C P^2$, one can show that
certain finite coverings for a given minimal Lagrangian immersion into $\C P^2$ have a
global horizontal lift. Below we consider similar questions for the situation considered in this paper.

The notion of ``horizontal lift'' has been defined in Definition \ref{def:horizontal}.
Whether $f$ is locally horizontally liftable can be decided by virtue of the following result.
\begin{lemma} \label{lem:horizontalLagrangian}
Let $f: M^n \to \DP$ be an immersion. Then locally around every point $y \in M^n$ there exists a horizontal lift $\f: U \rightarrow \CM$ of $f$ from a neighbourhood $U \subset M^n$ of $y$ if and only if the immersion $f$ is Lagrangian with respect to
the symplectic form $\omega$ on $\DP$.
\end{lemma}

\begin{proof}
Let $y \in M^n$ be arbitrary, let $U \subset M^n$ be a simply connected neighbourhood of $y$, and let $\f: U \to \CM$ be a lift of $f$. By Proposition \ref{prop:horizontal} the lift $\f$ can be scaled to a horizontal lift $\f_c: U \to \CM$ if and only if the form $\psi = \langle d_yx, \chi \rangle$ is exact on $U$. Since $U$ is simply connected, this is equivalent to the vanishing of the exterior derivative of $\psi$. We have $\psi = \sum_{i=1}^{n+1}\chi^id_yx^i$, and hence $d\psi = -\sum_{i=1}^{n+1}dx^i \wedge d\chi^i$.

 Let now $Z,Z' \in T_yM^n$ be arbitrary tangent vectors, and $d_y(Z) = (U,\mathcal{U})$, $d_y(Z') = (V,\mathcal V)$ their images in the tangent space to $\CM$ at $\f(y) = (x,\chi)$. By the above we have $d\psi(Z,Z') = -\langle U,\mathcal{V} \rangle + \langle V,\mathcal{U} \rangle = -2\hat{\omega}_{(x,\chi)}((U,\mathcal{U}),(V,\mathcal{V}))$. It follows that $d\psi$ vanishes if and only if the degenerate form $\hat\omega$ vanishes on the image of $d\f$. However, this condition is equivalent to the vanishing of the symplectic form $\omega$ on the image of $df$, or to the condition that $f$ is Lagrangian.
\end{proof}

Combining the lemma just above with Corollary \ref{aboutlifts} we obtain
\begin{prop}
 If $f:  M^n \rightarrow \DP$ is  a Lagrangian immersion, then
  there exists a two-fold cover $\pi: \hat{M}^n \to M^n$ such that the natural lift
  $\hat{f} = f \circ \pi$ of $f$ to $\hat{M}$ has a global  horizontal
 lift $\hat{\f} : \hat{M}^n \rightarrow \CM$.
 \end{prop}
 \begin{proof}
 It suffices to apply the construction in the proof of Lemma \ref{lem:horizontalLagrangian} to the lift $\hat\f$ in (3) of Corollary \ref{aboutlifts}.
 \end{proof}

It would be interesting to understand, what manifolds $M^n$ have horizontal lifts
for all immersions into $\DP$.
 At least locally it is the Lagrangian immersions. The lift is horizontal
 if and only if $\langle d_yx, \chi \rangle \equiv 0$.
 A horizontal lift of a immersion has a close relation to centro-affine immersions,
 that is, the position vector of an immersion transverses to the tangent plane.
\begin{prop}
 Assume  $f: M^n \to \DP$ has a horizontal lift $\f: M_n \rightarrow \CM$, $\f: y \mapsto (x,\chi)$.
 Then the following holds$:$
\begin{enumerate}
\setlength{\itemsep}{0cm}
\renewcommand{\labelenumi}{(\arabic{enumi})}
\item[{\rm (1)}] $x$ is a centro-affine immersion of $M$ into $\mathbb R^{n+1}$ such that $\chi$ is its co-normal map.
\item[{\rm (2)}] $\chi$ is a centro-affine immersion of $M$ into $\mathbb R_{n+1}$ such that $x$ is its co-normal map.
\item[{\rm (3)}] $f: M^n \to \DP$ is Lagrangian, i.e., for every two vector fields $X,Y$ on $M_n$ we have $\omega(f_*X,f_*Y) = 0$.
\end{enumerate}
\end{prop}

\begin{proof}
By \eqref{dectildef} $\f$ is horizontal if and only if $\langle d_yx(Z), \chi \rangle = \langle x, d_y\chi (Z) \rangle = 0$ for all $Z$.
Hence the tangent space of the immersion $x: M^n \to \mathbb R^{n+1}$ is orthogonal to $\chi$. But $\langle x,\chi \rangle \equiv 1$, and the first assertion follows.
The second assertion is proven similarly.

Let us prove the third assertion. Consider the 2-form $h(Z,W) = \langle d_yx(Z),d_y\chi(W) \rangle$ on $M^n$. By \cite[Proposition II.5.1]{NomizuSasaki} this form is proportional to the affine fundamental form induced by the centro-affine immersion $x$, and hence symmetric. Therefore its skew-symmetric part vanishes. In particular, if
 $d\f(Z) = (\widehat{\X},\widehat{\Xt})$, $d\f(W) = (\widehat{\Y},
 \widehat{\Yt})$ are tangent vector fields to the lift $\f$, then the form $\hat\omega(Z,W) = \frac12(\langle \widehat{\X},\widehat{\Yt} \rangle - \langle \widehat{\Y},\widehat{\Xt} \rangle)$ vanishes. But $\hat\omega(Z,W) = \omega(f_*Z,f_*W)$, and hence $f$ is Lagrangian.
\end{proof}
\begin{remark}
 Centro-affine immersions in affine space $\mathbb R^{n+1}$
 are important in affine differential geometry,
 see \cite{NomizuSasaki} for example.
\end{remark}

\section{$k$-symmetric spaces and primitive harmonic maps}\label{app:ksymmetric}
In \cite{DoKoMa19}, \cite{DFKW19}, $k$-symmetric spaces and primitive harmonic maps
have been considered. We will recall basic results for our case.
\subsection{The automorphism $\sigma_H$ and $k$-symmetric spaces} \label{subsc:realform}
 Let us consider the order $6$ automorphism $\sigma_H$ on $\mathfrak g = \sl$ given in
 \eqref{defsigma}, and also consider the order $6$ automorphism  $\sigma_H^G$
 on the connected Lie group $G =  \SL$ defined by
\begin{equation}\label{eq:sigmaonG}
\sigma_H^G(g) =  P_H^{\epsilon} (g^{T})^{-1}P_H^{\epsilon}.
\end{equation}
 By abuse of notation we will also denote $\sigma_H^G$ by $\sigma_H$.
 We then associate the order $3$ and $2$ autmorphisms $\sigma_H^2$ and $\sigma_H^3$
 as
\begin{align}\label{eq:sigma2}
\sigma_H^2 (X) &= P_2 X P_2^{-1},\quad \mbox{with} \quad
 P_2 = \di ( \epsilon^4, \epsilon^2, 1), \\
\label{eq:sigma3}
\sigma_H^3 ( X) & =  -P_3 X^T P_3,
\quad \mbox{with} \quad
P _3=
\begin{pmatrix}
0& 1 & 0\\
1& 0 & 0\\
0&0&-H\\
\end{pmatrix}.
\end{align}
 Note that $\sigma_H^2$ is an inner automorphism and $\sigma_H^3$ is an outer automorphism,
 respectively.
 A straightforward computation shows that
 the eigen-spaces of $\sigma_H$ can be computed as
\begin{align*}
\mathfrak{g}_0&=\left\{
                    \begin{pmatrix}
                    a &  &  \\
                     & -a &  \\
                     &  & 0 \\
                \end{pmatrix}
                \;\middle|\; a\in \C
\right\},
\quad
\mathfrak{g}_1=\left\{\begin{pmatrix}
                    0 & b & 0 \\
                     0 & 0 & a \\
                    -H a & 0 & 0 \\
                  \end{pmatrix}\;\middle|\; a,b\in\C
\right\},
\\
\mathfrak{g}_2&=\left\{
                  \begin{pmatrix}
                    0 & 0& a  \\
                     0& 0 & 0 \\
                     0& H a & 0 \\
                  \end{pmatrix} \;\middle|\; a\in \C
              \right\},
\quad
\mathfrak{g}_3=\left\{
                  \begin{pmatrix}
                    a &  &  \\
                      & a &  \\
                      &  & -2a \\
                  \end{pmatrix} \;\middle|\; a\in \C
              \right\},\\
\mathfrak{g}_4&=\left\{
                  \begin{pmatrix}
                    0 & 0 & 0  \\
                     0& 0 & a \\
                     Ha & 0 & 0 \\
                  \end{pmatrix} \;\middle|\; a\in \C
            \right\},
\quad
\mathfrak{g}_5=\left\{
                  \begin{pmatrix}
                    0 & 0 & -Ha \\
                     b & 0 & 0 \\
                     0 & a & 0 \\
                  \end{pmatrix} \;\middle|\; a, b\in \C
                \right\}.
\end{align*}
The eigen-spaces of $\sigma_H^2$ are given by
$\mathfrak{g}_1 + \mathfrak{g}_4$ for the eigenvalue  $\epsilon^2$,
$\mathfrak{g}_2 + \mathfrak{g}_5$ for the eigenvalue  $\epsilon^4$ and
 $\mathfrak{g}_3 + \mathfrak{g}_0$ for the eigenvalue  $1$.
Similarly the eigen-spaces for $\sigma_H^3$ are
 $\mathfrak{g}_4 + \mathfrak{g}_2 + \mathfrak{g}_0$ for the eigenvalue  $1$
 and
 $\mathfrak{g}_1 + \mathfrak{g}_3 + \mathfrak{g}_5$ for the eigenvalue  $\epsilon^3 =  -1$.
 It is important that the real form involution $\tau$ in \eqref{eq:tau} and the
 order $6$ automorphism $\sigma_H$ in \eqref{defsigma} commute, i.e.,
 $\tau \circ \sigma_H =  \sigma_H \circ \tau$ holds.
 Therefore $\tau$ and the eigen-spaces of $\sigma_H$ obey the relation
\begin{equation}\label{eq:tausigma}
\tau (\mathfrak g_j) = \mathfrak g_{-j}, \quad j=0, 1, \dots, 5.
\end{equation}
 In particular, $\mathfrak g_0$ and $\mathfrak g_0 \oplus \mathfrak g_3$
 are subalgebras of $\slr$ with the obvious complexifications.

 Since $\sigma_H$ and $\tau$ commute, we can give a definition of $k$-symmetric spaces
 as follows.
 \begin{definition}\label{def:k-symmetric}
 Let $G^{\R}/G^{\R}_0$ be a real homogeneous space such that $G^{\R}$ is a
 real form of a complex Lie group $G$ given by a real form involution
 $\tau$, that is, $G^{\R} = \Fi (G, \tau)$.
 Moreover,  let $\sigma$
 be an order $k \;(k\geq 2)$ automorphism  of $G,$ leaving  $G^{\R}$ invariant and commuting with $\tau$.
 Then  $G^{\R}/G^{\R}_0$ is called a \emph{$k$-symmetric space} if
 the following condition is satisfied
\begin{equation}
 \Fi(G^{\R},\sigma)^\circ
 \subset G^{\R}_0 \subset \Fi (G^{\R}, \sigma),
\end{equation}
 where $\Fi(G^{\R}, \sigma)^\circ$ denotes the identity component of $\Fi(G^{\R},
\sigma)$.
\end{definition}

\subsection{Primitive maps and the extended frames}
We now consider a complex Lie group as before and let $\tau$ denote an
anti-holomorphic involution of $G$, and set
\begin{equation*}
G^\R = \Fi(G,\tau) \quad \mbox{and}\quad \operatorname{Lie}G^\R = \mathfrak{g}^\R.
\end{equation*}
\begin{definition}\label{def:primitiveharmonic}
Let $\kappa$ be any automorphism of $\mathfrak{g}$ of finite order $k > 2$.
Let $\mathfrak{g}_m$ denote the  eigen-spaces of $\kappa$, where we choose $m \in \mathbb{Z}$ and actually work with $m \mod k$.
Let $\mathcal{F} : \D \rightarrow G$ be a smooth map. Then $\mathcal{F}$ will be called \emph{primitive relative to $\kappa$} if
\begin{equation*}
\mathcal{F}^{-1} d\mathcal{F} = \alpha_{-1}  dz + \alpha_0^{\prime} dz+\alpha_0^{\prime\prime} d\bar{z}+ \alpha_1 d\bar{z} \in \mathfrak{g}_{-1} +  \mathfrak{g}_0 + \mathfrak{g}_1,
\end{equation*}
where  $\alpha_m$, $\alpha_0^{\prime}$ and $\alpha_0^{\prime\prime}$ take values in
an eigen-space
$\mathfrak{g}_m$ of  $\kappa$.
\end{definition}
By abuse of notation we will also write $\alpha_0 =  \alpha_0^\prime dz + \alpha_0^{\prime \prime}d \bar{z}$.
\begin{lemma}
Let $\mathcal{F}$  be primitive relative to $\kappa$ and let us write
$\mathcal{F}^{-1} d\mathcal{F} = \alpha_{-1} dz+ \alpha_0 + \alpha_1 d \bar{z}$.
Then
$\lambda^{-1} \alpha_{-1} dz+ \alpha_0 + \lambda \alpha_1d \bar{z}$
is integrable for all $\lambda \in \C^*$.
\end{lemma}
\begin{proof}
Together with a  straightforward computation one needs to use that because of $k>2$ the sum $\mathfrak{g}_{-1} + \mathfrak{g}_0 + \mathfrak{g}_1$ of eigen-spaces is direct.
\end{proof}
The importance of this observation has been elaborated on and explained in
\cite[Section 3.2]{BP} and \cite{B}.
\begin{theorem}[\cite{BP, B}] \label{theoremharmonicity}
Let $G$ be a complex Lie group,   $\sigma_H$ an automorphism of $G$ of finite order $k \geq 2$
and $\tau$ an anti-holomorphic involution of $G$ which commutes with  $\sigma_H$.
Let $G^\R_0$ be any Lie subgroup of $G^\R$ satisfying
 $\Fi (G^\R, \sigma_H)^{\circ} \subset {G^\R_0} \subset \Fi (G^\R, \sigma_H)$.
Then we consider the $k$-symmetric space $G^\R / {G^\R_0}$ together with the (pseudo-)Riemannian structure induced by some bi-invariant (pseudo-)Riemannian structure on $G^\R$.
Let $h:\D \rightarrow G^\R / {G^\R_0}$ be a 
smooth map and
$\mathcal{F} : \D \rightarrow G^\R$ a frame for $h$, i.e., $h = \pi \circ \mathcal{F}$, where
$\pi: G^\R \rightarrow G^{\R}/{G^\R_0}$ denotes the canonical projection.

Then the following statements hold$:$
\begin{enumerate}
\setlength{\itemsep}{0cm}
\renewcommand{\labelenumi}{(\arabic{enumi})}
 \item[{\rm (1)}]   If $k=2$, then $h$ is harmonic if and only if  $\lambda^{-1} \alpha_{-1} dz+ \alpha_0 + \lambda \alpha_1 d\bar{z}$ is integrable for all $\lambda \in \C^*$.

\item[{\rm (2)}] If $k>2$, then $h$ is harmonic  if $\mathcal{F}$ is primitive relative
to  $ \sigma_H$.
\end{enumerate}
\end{theorem}
 From the above theorem, we have the following definition.
\begin{definition}
 Retain the notation in Theorem \ref{theoremharmonicity}.
\begin{enumerate}
\setlength{\itemsep}{0cm}
\renewcommand{\labelenumi}{(\arabic{enumi})}
 \item
 The frame $\mathcal F$ is called \emph{primitive harmonic}, if
 $\mathcal{F}^{-1} d\mathcal{F} =
 \alpha_{-1}dz  + \alpha_0 + \alpha_1 d \bar{z} $ such that
 $\lambda^{-1} \alpha_{-1} dz+ \alpha_0 + \lambda \alpha_1 d\bar{z}$
 is integrable for all $\lambda \in \C^*$.
\item  The map $h$ is called
 \emph{primitive harmonic map}, if the frame $\mathcal F$ is
 primitive harmonic.
\end{enumerate}
\end{definition}
This admits a direct application of the \emph{loop group method}, see \cite{DPW}.
Since $\tau$ maps $\mathfrak{g}_m$ to $\mathfrak{g}_{-m}$, we can assume that
$\mathcal{F}_\lambda$ is contained in $G^\R$ for all $\lambda \in S^1$.
 We will usually also assume $\mathcal{F}(z_0,\lambda) = I$ for a once and for all fixed base point $z_0$.

Then it follows from the above that also $h_\lambda = \mathcal{F}_\lambda \mod G^\R_0$ is a primitive harmonic map with frame $\mathcal{F}_\lambda$.
Usually  $\mathcal{F}_\lambda$ is called \emph{an extended frame} for $h$.

\section{Various bundles}\label{app:bundles}
 This section is an adaption of Section 3 in \cite{DoKoMa19} to
 our case.
 In \cite[Section 3]{DoKoMa19}, three $6$-symmetric spaces
 of dimension $7$ which are bundles over $S^5$ were defined.
 We analogously define three $6$-symmetric spaces
 of dimension $7$ which are bundles over $\CMf$, $FL_1^H$, $FL_2^H$, and $FL_3^H$.

 Recall the para-hermitian inner product of $\PC^3$
 with a para-Hermitian form
\begin{equation}
 \langle u, v \rangle_h = u^{*T} P_H v, \quad
  P_H =
\begin{pmatrix}
 0 & 1 & 0 \\
 1 & 0 & 0 \\
 0 & 0 & -H
\end{pmatrix},
\end{equation}
 where $H=1$ (resp. $H=-1$) for the ellipic (resp. hyperbolic) case
 and $*$ denotes the para-complex conjugate of a paracomplex vector in $\PC^3$.
 It is invariant under the transformation
\begin{equation}\label{eq:action}
(\PC)^3 \ni u \mapsto \left( \frac{1+i^{\prime}}{2}A + \frac{1-i^{\prime}}{2}A^{-T} \right)u \in (\PC)^3,
\quad
 A \in \GLR.
\end{equation}
 We first choose a natural basis  of $\PC^3$:
 \[
  e_1 = (1, 0, 0)^T,\quad  e_2 = (0, 1, 0)^T,\quad  e_3= (0, 0, 1)^T.
 \]

 $(1) \; FL_1^H:$
 For the real $6$-dimensional symplectic vector space
 $\PC^3$ given by the symplectic form $\omega = - \Im \langle \;,\; \rangle_h$,
 the family of (real) oriented Lagrangian subspaces of  $\PC^3$
 forms a submanifold of the manifold of real Grassmannian $3$-spaces of $\PC^3$.
 They are the  \emph{ Grassmannian manifold} $\LGr(3, \PC^3)$ of oriented Lagrangian subspaces.  It is easy to see that
 $\LGr(3, \PC^3)$ can be represented as the homogeneous space
 $\GLR / \O$.
 The special orthogonal matrix group $\SO$ as the
 connected subgroup of $\SLR$
 corresponding to the sub-Lie-algebra  of $\slr$  given by
\[
 \so = \left\{
 \begin{pmatrix}
  ia & 0 & \sqrt{-H} b \\
 0 & -ia & \sqrt{-H} \bar b \\
 -\sqrt{-H}^{-1}\bar b & -\sqrt{-H}^{-1}b  & 0
 \end{pmatrix}\;\middle|\;
 a \in \R, b \in  \C
\right\}\subset  \slr,
\]
 where $H=1$ (resp. $H=-1$) for the ellipic (resp. hyperbolic) case,
 which is isomorphic to the standard $\soStan$ by
 the automorphism $X \mapsto \Ad (R_H)(X)$ with
\[
 R_H=\begin{pmatrix}
\frac1{\sqrt{2}} & \frac{1}{\sqrt{
2}}&0 \\
\frac{i}{\sqrt{2}}&-\frac{i}{\sqrt{2}}&0 \\
 0&0& \sqrt{-H}
 \end{pmatrix}.
\]
 The orbit of $\SLR$ in $\LGr(3, \PC^3)$ through the point $e \in \SO$ will be  called
 the \emph{special Lagrangian Grassmannian} and it will be denoted by
 $\SLGr(3, \PC^3)$.
 The elements in this orbit will be called \emph{oriented  special Lagrangian
subspaces} of $\PC^3$. Thus we have the following:
\begin{prop}\label{Prp:Grass}
  $\SLR$ acts transitively on  $\SLGr(3, \PC^3)$,
  and we obtain
 \[
\SLGr(3, \PC^3) = \SLR / \SO.
 \]
The base point $e\in \SO$ corresponds to the real Lagrangian subspace of $\PC^3$ given by
 $R_H^{-1} \R^3$.
\end{prop}
We now define a bundle
\begin{equation*}
FL_1^H = \{ (v,V)\;|\; v \in \CMf, \; v \in V, \;
 V  \in  \SLGr(3, \PC^3)\}.
\end{equation*}
 It is easy to verify that $\SLR$ acts (diagonally) on $FL_1^H$.
 Note that the natural projection from $FL_1^H$ to $\DP$ is a
 pseudo-Riemannian submersion which is equivariant under the natural group actions.
 Since $\CMf = \SLR/\SLt$, where  $\SLt$ means
 $\SLt \times \{1\}$,
 the stabilizer at
 \[
  (e_3, \operatorname{span}_{\R}\{e_1, e_2,  e_3 \}) \in FL_1^H
 \]
  is clearly given by
 $\SLt \cap \SO$, that is
 \[
  \Uone =\{(a, a^{-1}, 1)\;|\; a \in S^1\}.
 \]
 Therefore
\[
 FL_1^H = \SLR/\Uone.
\]

 $(2)\; FL_2:$
 To define $FL_2^H$,
 we consider certain \emph{special regular para-complex flags} in $\PC^3$.
 A regular para-complex flag  $\mathcal{Q}$ is a sequence of four
 para-complex subspaces,
 $Q_0 = \{0\} \subset Q_1 \subset Q_2 \subset  Q_3 = \PC^3$
 of $\PC^3$,
 where $Q_j$ has  para-complex dimension $j$.
 We then define the notion of a \emph{special regular para-complex flag}
 in $\PC^3$ over $q \in \CMf$
 by requiring that we have a  regular complex flag in $\PC^3$,
 where the space $Q_1$  satisfies $Q_1 = \PC q$.
 Thus we define
 \begin{equation*}
FL_2^H =
 \left\{ (w,\mathcal{W})\;\middle|  \;
\begin{array}{l}
\text{$w \in \CMf$, $\mathcal W$ is
 a special regular para-complex}  \\
\text{ flag over $w$ in $\PC^3$ satisfying $W_1 = \PC w$}
\end{array}
 \right\}.
\end{equation*}
  The definition of a special flag means that  for a given vector $q \neq 0$ in $\PC^3$ one can find three pairwise orthogonal vectors $q_1, q_2, q_3 \in \PC^3$ with
 $q_3 = \frac{q}{|q|}$ such that
the vectors $q_1,q_2$ and $q_3$ represent the same orientation as $e_1, e_2, e_3$.
 By an argument similar to the previous case we conclude
 that  $\SLR$  acts transitively on the family of special flags.
 Moreover, the stabilizer of the action at the point
 $(e_3, 0 \subset \PC e_3 \subset  \PC e_3 \oplus \PC e_2
 \subset  \PC e_3 \oplus \PC e_2  \oplus \PC e_1)$
 is again given by $\SO \cap \di$, where $\di$ denotes the set of all
 diagonal matrices in $\SLR$. Thus it is again $\Uone$ and we have altogether shown
\begin{prop}\label{Prp:Flags}
 $\SLR$ acts transitively on $FL_2^H$, and $FL_2^H$ can be
 represented as
\[
  FL_2^H = \SLR / \Uone.
\]
 \end{prop}
 Note that the natural projection from $FL_2^H$ to $\DPt$
 is a pseudo-Riemannian submersion which is equivariant under the natural group actions.

 $(3)\; FL_3^H:$ Finally, using the isometry group $\SLR$ of $\CMf$,
 we can directly define a homogeneous space $FL_3^H$ as
\begin{equation}\label{eq:FL3}
FL_3^H = \left\{ U P_H^{\epsilon} \;U^T\;\middle|\; \mbox{$U \in \SLR$ and $P_H^{\epsilon}
 =\begin{pmatrix}
   0 & \epsilon^2 & 0 \\
   \epsilon^4 &0 & 0 \\
   0 & 0 & -H
  \end{pmatrix}$}\right\},
\end{equation}
 where $\epsilon = e^{\pi i/3}$ and $H =\pm 1$.
\begin{theorem}\label{Thm:3.3}
We  retain the assumptions and the notion above.
 Then the following statements hold$:$
\begin{enumerate}
\setlength{\itemsep}{0cm}
\renewcommand{\labelenumi}{(\arabic{enumi})}
\item[{\rm (1)}] The spaces $FL_j^H$  $(j = 1,2,3)$ are homogeneous under the natural action of
 $\SLR$.
\item[{\rm (2)}] The homogeneous space $FL_j^H$ $(j = 1,2,3)$ can be represented as
 \[
       FL_j^H = \SLR/ \Uone,
 \quad \mbox{where}\quad
\Uone = \{ \di (a,a^{-1}, 1)\;|\;
 a\in S^1\}.
 \]
 In particular they are all $7$-dimensional.
\end{enumerate}
\end{theorem}
\begin{proof}
The statements clearly follow from Proposition \ref{Prp:Grass},
 Proposition \ref{Prp:Flags} and  the definition of $FL_3^H$ in \eqref{eq:FL3},
where  the stabilizer at $P_H^{\epsilon}$ is easily computed as $\Uone$.
\end{proof}


\begin{corollary}\label{cor:3.4}
 The homogeneous spaces $FL_j^H \;(j=1, 2, 3)$ are $6$-symmetric spaces.
Furthermore, they are naturally equivariantly diffeomorphic.
\end{corollary}
\begin{proof}
 First we note that the group $G^{\R} = \SLR$ has the
 complexification $G = \rm{SL}_3 \C$ and is the fixed point set group of the
 real form involution $\tau$ given in \eqref{eq:tauonG0}.

 We show that $FL_3$ is a $6$-symmetric space.
 First note that the stabilizer
 \begin{equation}\label{eq:stabFL3}
 \textrm{Stab}_{P} = \{ X \in \SLR \;|\; X P \; X^T = P\}
 \end{equation}
 at the point $P$ of $FL_3$
 is $\Uone$.
We already know  that the order $6$-automorphism $\sigma_H$
 of $\SLR$ given
 in \eqref{eq:sigmaonG} and the real form involution $\tau$
 commute.
 Moreover,  a direct computation shows that the fixed point set of
 $\sigma_H$ in $\SLR$ is $\Uone$.
 Thus $\textrm{Stab}_{P}$ satisfies
 the condition in Definition  \ref{def:k-symmetric}.
Hence  $FL_3^H$ is $6$-symmetric space  in the sense of Definition
 \ref{def:k-symmetric}.
 Furthermore, since all the spaces $FL_j^H$ are $\SLR$-orbits with the same stabilizer, the identity homomorphism of $\SLR$ descends for any pair of homogeneous spaces
 $FL_j^H$ and $FL_m^H$
 to a diffeomorphism
 \[
  \phi_{jm}: FL_m^H \rightarrow FL_j^H
 \]
 such that for any $g \in \SLR$ and $p\in FL_m^H$  we have
\[
 \phi_{jm} (g.p) = g.\phi_{jm} (p).
\]
 As a consequence, also $FL_1^H$ and $FL_2^H$ are  $6$-symmetric spaces.
\end{proof}
 We have seen that the homogeneous spaces $FL_j^H$ $(j=1, 2, 3)$ are
 $7$ dimensional $6$-symmetric spaces. In this section we define
 natural projections from $FL_j^H$ to several homogeneous spaces.

 First from $FL_1^H$, we have a projection to $\SLGr(3, \PC^3)$ given by
\[
 FL_1^H \ni (v, V) \longmapsto V \in  \SLGr(3, \PC^3).
\]
 It is easy to see that $\SLGr(3, \PC^3)$ is a symmetric space with
 the involution $\sigma_H^3$ defined in \eqref{eq:sigma3}.

 Next from $FL_2^H$, we have a projection to a full flag manifold:
\[
  FL_2^H \ni (w, W) \longmapsto W \in  Fl_2^H,
\]
 where $Fl_2^H$ is defined as
\[
 Fl_2^H = \{\mathcal W \mid \mbox{$\mathcal W$ is a regular para-complex flag in $\PC^3$} \}.
\]
 It is easy to see that $Fl_2^H$ is a $3$-symmetric space with
 the involution $\sigma_H^2$ stated  in \eqref{eq:sigma2}.

 Finally from $FL_3^H$, we have two projections. We first
 let $k \in \textrm{Stab}_{P_H^{\epsilon}}$ as in \eqref{eq:stabFL3} with
\[
P_H^{\epsilon}=
  \begin{pmatrix}
 0  & \epsilon^2 & 0 \\
 \epsilon^4 & 0 & 0 \\
0  & 0  & 1
 \end{pmatrix}, \quad \epsilon = e^{\pi i /3},
\]
  then a straightforward computation shows that
\begin{gather*}
 k P_H^{\epsilon} (P_H^{\epsilon})^T k^{-1} =   k P_H^{\epsilon} k^T (P_H^{\epsilon})^T    = P_H^{\epsilon}(P_H^{\epsilon})^T, \\
 k P_H^{\epsilon} (P_H^{\epsilon})^T P_H^{\epsilon}  k^{T} =  P_H^{\epsilon} (k^T)^{-1} (P_H^{\epsilon})^T k^{-1} P_H^{\epsilon}  =  P_H^{\epsilon}(P_H^{\epsilon})^T P_H^{\epsilon}.
\end{gather*}
 Therefore we have two projections
\begin{eqnarray*}
 &FL_3^H \ni U P_H^{\epsilon} U^T \longmapsto U (P_H^{\epsilon} (P_H^{\epsilon})^T) U^{-1} \in \widetilde {Fl_2^H}, \\
 &FL_3^H \ni U P_H^{\epsilon} U^T \longmapsto U (P_H^{\epsilon} (P_H^{\epsilon})^T P_H^{\epsilon}) U^T \in \widetilde \SLGr(3, \PC^3),
\end{eqnarray*}
 where $\widetilde {Fl_2^H}$ and $\widetilde \SLGr(3, \PC^3)$ are defined as
 \[
\widetilde {Fl_2^H} = \{U (P_H^{\epsilon} (P_H^{\epsilon})^T) U^{-1}\;|\; U \in \SLR\}, \quad
\widetilde \SLGr(3, \PC^3) = \{U (P_H^{\epsilon} (P_H^{\epsilon})^T P_H^{\epsilon}) U^{T}\;|\; U \in \SLR\}.
 \]
 Note that it is easy to compute
 \[
  P_H^{\epsilon} (P_H^{\epsilon})^T =
\begin{pmatrix}
 \epsilon^4 & 0 & 0 \\
 0 & \epsilon^2 & 0 \\
 0 & 0 & 1
\end{pmatrix}, \quad
  P_H^{\epsilon} (P_H^{\epsilon})^T P_H^{\epsilon} =P_H
\begin{pmatrix}
 0 & 1 & 0 \\
 1 & 0 & 0 \\
 0 & 0 & -H
\end{pmatrix},
 \]
 and the stabilizer in $\SLR$ at $P_H^{\epsilon} (P_H^{\epsilon})^T$ of $\widetilde{Fl_2^H}$ and
 the stabilizer in $\SLR$ at $P_H^{\epsilon} (P_H^{\epsilon})^T P_H^{\epsilon} $ of $\widetilde{\SLGr}(3, \PC^3)$ are
\[
\textrm{Stab}_{P_H^{\epsilon}(P_H^{\epsilon})^T} = \textrm{D}_3, \quad
 \textrm{Stab}_{P_H^{\epsilon}(P_H^{\epsilon})^TP_H^{\epsilon}} = \SO,
\]
 where
\[
 \textrm{D}_3 =
\{ \di (a_1, a_2, a_3) \in \SLR \},
\]
and where $\textrm{Stab}_{PP^TP}$ is exactly the same group as the stabilizer
of  $\SLGr(3, \PC^3)$. Thus  $\SLGr(3, \PC^3)$ and $\widetilde{\SLGr}(3, \PC^3)$ are naturally
equivariantly diffeomorphic. An analogous argument applies to  $Fl_2^H$ and  $\widetilde{Fl_2^H}$. Now  the stabilizer of $\widetilde{Fl_2^H}$ is determined by  the matrix characterizing $\sigma_H^2$, whence $\widetilde{Fl_2^H}$ (and thus $Fl_2^H$) is the $3$-symmetric space associated with $\sigma_H^2$. Similarly, $\SLGr(3, \PC^3)$ (and thus $\widetilde{\SLGr}(3, \PC))$
is the symmetric space associated with $\sigma_H^{3}$.
\\[0.5cm]

\textbf{Data Availability} Data sharing not applicable to this article as no datasets were generated or analysed during the current study.

\textbf{Conflict of interest} The authors state that there is no conflict of interest.


\end{document}